\documentclass[11pt]{article}
\usepackage{amsmath,amsthm,amssymb}
\usepackage[usenames,dvipsnames]{xcolor}
\usepackage{enumerate}
\usepackage{graphicx}
\usepackage{cite}
\usepackage{comment}
\usepackage{oands}
\usepackage{tikz}
\usepackage{changepage}
\usepackage{bbm}
\usepackage{mathtools}
\usepackage[margin=1in]{geometry}
\usepackage[pagewise,mathlines]{lineno}
\usepackage{appendix}
\usepackage{stmaryrd}
\usepackage{multicol}
\usepackage{microtype}
\usepackage[colorinlistoftodos]{todonotes}

\newcommand{\ewain}[1]{{\color{blue}{#1}}}

\usepackage[pdftitle={Up-to-constants comparison of Liouville first passage percolation and Liouville quantum gravity},
  pdfauthor={Jian Ding and Ewain Gwynne},
colorlinks=true,linkcolor=NavyBlue,urlcolor=RoyalBlue,citecolor=PineGreen,bookmarks=true,bookmarksopen=true,bookmarksopenlevel=2,unicode=true,linktocpage]{hyperref}

\theoremstyle{plain}
\newtheorem{thm}{Theorem}[section]

\newtheorem{lem}[thm]{Lemma}
\newtheorem{prop}[thm]{Proposition}

\def\@rst #1 #2other{#1}
\newcommand\MR[1]{\relax\ifhmode\unskip\spacefactor3000 \space\fi
  \MRhref{\expandafter\@rst #1 other}{#1}}
\newcommand{\MRhref}[2]{\href{http://www.ams.org/mathscinet-getitem?mr=#1}{MR#2}}

\theoremstyle{definition}
\newtheorem{defn}[thm]{Definition}
\newtheorem{remark}[thm]{Remark}

\numberwithin{equation}{section}

\newcommand{\dsb}{\begin{adjustwidth}{2.5em}{0pt}
\begin{footnotesize}}
\newcommand{\dse}{\end{footnotesize}
\end{adjustwidth}}

\newcommand{\ssb}{\begin{adjustwidth}{2.5em}{0pt}}
\newcommand{\sse}{\end{adjustwidth}}

\newcommand{\aryb}{\begin{eqnarray*}}
\newcommand{\arye}{\end{eqnarray*}}
\def\alb#1\ale{\begin{align*}#1\end{align*}}
\def\allb#1\alle{\begin{align}#1\end{align}}
\newcommand{\eqb}{\begin{equation}}
\newcommand{\eqe}{\end{equation}}
\newcommand{\eqbn}{\begin{equation*}}
\newcommand{\eqen}{\end{equation*}}

\newcommand{\BB}{\mathbbm}
\newcommand{\ol}{\overline}

\newcommand{\op}{\operatorname}

\newcommand{\frk}{\mathfrak}
\newcommand{\eqD}{\overset{d}{=}}
\newcommand{\ep}{\varepsilon}
\newcommand{\rta}{\rightarrow}

\newcommand{\wt}{\widetilde}
\newcommand{\wh}{\widehat}
\newcommand{\mcl}{\mathcal}

\newcommand{\bdy}{\partial}

\newcommand{\ccM}{{\mathbf{c}_{\mathrm M}}}

\let\originalleft\left
\let\originalright\right
\renewcommand{\left}{\mathopen{}\mathclose\bgroup\originalleft}
\renewcommand{\right}{\aftergroup\egroup\originalright}

\title{Up-to-constants comparison of Liouville first passage percolation and Liouville quantum gravity}
 \date{ }
 \author{
\begin{tabular}{c} Jian Ding\\[-5pt]\small University of Pennsylvania \end{tabular}
\begin{tabular}{c} Ewain Gwynne\\[-5pt]\small University of Chicago \end{tabular}
}

\begin{document}

\maketitle

\begin{abstract}
\emph{Liouville first passage percolation (LFPP)} with parameter $\xi > 0$ is the family of random distance functions $\{D_h^\epsilon\}_{\epsilon >0}$ on the plane obtained by integrating $e^{\xi h_\epsilon}$ along paths, where $\{h_\epsilon\}_{\epsilon >0}$ is a smooth mollification of the planar Gaussian free field. 
Recent works have shown that for all $\xi > 0$ the LFPP metrics, appropriately re-scaled, admit non-trivial subsequential limiting metrics. In the case when $\xi < \xi_{\mathrm{crit}} \approx 0.41$, it has been shown that the subsequential limit is unique and defines a metric on $\gamma$-Liouville quantum gravity $\gamma = \gamma(\xi) \in (0,2)$. 

We prove that for all $\xi > 0$, each possible subsequential limiting metric is nearly bi-Lipschitz equivalent to the LFPP metric $D_h^\epsilon$ when $\epsilon$ is small, even if $\epsilon$ does not belong to the appropriate subsequence. Using this result, we obtain bounds for the scaling constants for LFPP which are sharp up to polylogarithmic factors. We also prove that any two subsequential limiting metrics are bi-Lipschitz equivalent. 

Our results are an input in subsequent work which shows that the subsequential limits of LFPP induce the same topology as the Euclidean metric when $\xi = \xi_{\mathrm{crit}}$. We expect that our results will also have several other uses, e.g., in the proof that the subsequential limit is unique when $\xi \geq \xi_{\mathrm{crit}}$.
\end{abstract}

\tableofcontents

\section{Introduction}

\subsection{Liouville first passage percolation}
\label{sec-lfpp}
 
Let $h$ be the whole-plane Gaussian free field (GFF), normalized so that its average over the unit circle is zero. This means that $h$ is the Gaussian process on $\BB C$ with covariances given by
\eqbn
\op{Cov}(h(z) , h(w))=  \log \frac{\max\{|z|,1\}  \max\{|w|,1\}}{|z-w|} ,\quad\forall z,w\in \BB C ;
\eqen 
see~\cite[Section 2.1.1]{vargas-dozz-notes}.
The GFF does not make sense as a random function, but it can be defined as a random generalized function, meaning that we can define its integral against a smooth test function with sufficiently fast decay at $\infty$. We refer to the expository articles~\cite{shef-gff,berestycki-lqg-notes,pw-gff-notes} for more on the GFF. 

Recent works have shown that for $\xi  >0$, one can construct a random metric on $\BB C$ which is heuristically obtained by ``weighting lengths of paths by $e^{\xi h}$, then taking an infimum". The motivation for considering such a metric comes from the theory of Liouville quantum gravity (LQG). 
The reason for the quotations is that $h$ is a generalized function, not a true function, so $e^{\xi h}$ does not make literal sense. 
Consequently, to construct this metric one needs to take a limit of a family of approximating metrics called \emph{Liouville first passage percolation} (LFPP), which we discuss just below.
The goal of this paper is to prove a quantitative estimate for how close the LFPP metrics are to the limiting metric, which neither implies nor is implied by the convergence. As consequences of this estimate, we will deduce several other estimates for LFPP and LQG which are needed in future works. 

Let us now discuss the definition of LFPP. 
For $t  > 0$ and $z\in\BB C$, we define the heat kernel $p_t(z) := \frac{1}{2\pi t} e^{-|z|^2/2t}$ and we denote its convolution with $h$ by
\eqb \label{eqn-gff-convolve}
h_\ep^*(z) := (h*p_{\ep^2/2})(z) = \int_{\BB C} h(w) p_{\ep^2/2} (z  - w) \, dw  ,\quad \forall z\in \BB C  
\eqe
where the integral is interpreted in the sense of distributional pairing. The reason for integrating against $p_{\ep^2/2}(z,w)$ instead of $p_\ep(z,w)$ is so that the variance of $h_\ep^*(z)$ is of order $\log\ep^{-1} + O_\ep(1)$. 

For a parameter $\xi > 0$, we define the $\ep$-\emph{Liouville first passage percolation} (LFPP) metric associated with $h$ by
\eqb \label{eqn-gff-lfpp}
D_{h}^\ep(z,w) := \inf_P \int_0^1 e^{\xi h_\ep^*(P(t))} |P'(t)| \,dt ,\quad \forall z,w\in\BB C 
\eqe
where the infimum is over all piecewise continuously differentiable paths $P : [0,1]\rta \BB C$ from $z$ to $w$. We will be interested in (subsequential) limits of the re-normalized metrics $\frk a_\ep^{-1} D_h^\ep$, where the normalizing constant is defined by
\eqb \label{eqn-gff-constant}
\frk a_\ep := \text{median of} \: \inf\left\{ \int_0^1 e^{\xi h_\ep^*(P(t))} |P'(t)| \,dt  : \text{$P$ is a left-right crossing of $[0,1]^2$} \right\}      .
\eqe  
Here, by a left-right crossing of $[0,1]^2$ we mean a piecewise continuously differentiable path $P : [0,1]\rta [0,1]^2$ joining the left and right boundaries of $[0,1]^2$. 

The scaling constants $\frk a_\ep$ are not known explicitly, but it is shown in~\cite[Proposition 1.1]{dg-supercritical-lfpp} that for each $\xi > 0$, there exists $Q = Q(\xi) > 0$ such that 
\eqb  \label{eqn-Q-def}
\frk a_\ep = \ep^{1 - \xi Q  + o_\ep(1)} ,\quad \text{as} \quad \ep \rta 0 . 
\eqe 
We call $Q$ the \emph{LFPP distance exponent}. The existence of $Q$ is proven using a subadditivity argument, so its value is not known except that $Q(1/\sqrt 6) = 5/\sqrt 6$.\footnote{The case when $\xi = 1/\sqrt 6$ corresponds to Liouville quantum gravity with $\gamma=\sqrt{8/3}$, and the relation $Q(1/\sqrt 6) = 5/\sqrt 6$  is equivalent to the statement that the Hausdorff dimension of $\sqrt{8/3}$-LQG is equal to 4. See~\cite{dg-lqg-dim} for more detail.}
However, reasonably good rigorous upper and lower bounds for $Q$ in terms of $\xi$ are available~\cite{dg-lqg-dim,gp-lfpp-bounds,ang-discrete-lfpp}.

LFPP undergoes a phase transition at the critical parameter value 
\eqb \label{eqn-xi_crit}
\xi_{\op{crit}} := \inf\{\xi > 0 : Q(\xi) = 2\} .
\eqe
We do not know $\xi_{\op{crit}}$ explicitly, but the bounds from~\cite[Theorem 2.3]{gp-lfpp-bounds} give the approximation $\xi_{\op{crit}}  \in [0.4135 , 0.4189]$.  
 
\begin{defn}
We refer to LFPP with $\xi <\xi_{\op{crit}}$, $\xi = \xi_{\op{crit}}$, and $\xi > \xi_{\op{crit}}$ as the \emph{subcritical}, \emph{critical}, and \emph{supercritical} phases, respectively.
\end{defn}

We now briefly discuss what happens in each of the three phases. 
In the subcritical phase, it was shown by Ding, Dub\'edat, Dunlap, and Falconet~\cite{dddf-lfpp} that the re-scaled LFPP metrics $\frk a_\ep^{-1} D_h^\ep$ admit subsequential limits in probability with respect to the topology of uniform convergence on compact subsets of $\BB C\times \BB C$. Moreover, every possible subsequential limit is a metric which induces the same topology on $\BB C$ as the Euclidean metric. 
Later, it was shown by Gwynne and Miller~\cite{gm-uniqueness} (building on~\cite{local-metrics,lqg-metric-estimates,gm-confluence}) that the subsequential limit is uniquely characterized by a certain list of axioms, so $\frk a_\ep^{-1} D_h^\ep$ converges as $\ep\rta 0$, not just subsequentially. 

The limiting metric can be viewed as the distance function associated with \emph{$\gamma$-Liouville quantum gravity} (LQG) for an appropriate value of $\gamma  = \gamma(\xi) \in (0,2)$. In other words, we interpret the limit of subcritical LFPP as the the Riemannian distance function for the Riemannian metric tensor
\eqbn
e^{\gamma h(x,y)} \, (dx^2 + dy^2) 
\eqen
where $dx^2 + dy^2$ is the Euclidean metric tensor. The relationship between $\gamma$ and $\xi$ is given by either of two equivalent, but non-explicit, formulas:
\eqb \label{eqn-gamma-xi}
Q(\xi) = \frac{2}{\gamma}  +\frac{\gamma}{2} \quad \text{or, equivalently} \quad \gamma = \xi d(\xi) 
\eqe
where $d(\xi) > 2$ is the Hausdorff dimension of $\BB C$, equipped with the limiting metric. Neither of the formulas~\eqref{eqn-gamma-xi} gives an explicit relationship between $\xi$ and $\gamma$ since neither $Q(\xi)$ nor $d(\xi)$ is known explicitly. See~\cite{dg-lqg-dim} for further discussion. 

In the supercritical and critical phases, we showed in~\cite{dg-supercritical-lfpp} that the metrics $\frk a_\ep^{-1} D_h^\ep$ admit non-trivial subsequential limits with respect to the topology on lower semicontinuous functions on $\BB C\times \BB C$, which we recall in Definition~\ref{def-lsc} below. 
Every subsequential limit is a metric (not just a pseudometric). We expect that the subsequential limit is unique, but this has not yet been proven. 

In the supercritical case, if $D_h$ is a subsequential limiting metric, then $D_h$ does not induce the Euclidean topology on $\BB C$. Rather, there is an uncountable, dense, Lebesgue measure zero set of \emph{singular points} $z\in \BB C$ such that 
\eqb \label{eqn-singular-pt}
D_h(z,w) = \infty,\quad \forall w\in\BB C\setminus\{z\} .
\eqe
Roughly speaking, the singular points correspond to the points $z\in\BB C$ for which~\cite[Proposition 1.11]{pfeffer-supercritical-lqg} 
\eqbn
\limsup_{\ep\rta 0} \frac{ h_\ep(z) }{ \log\ep^{-1} }  > Q , 
\eqen
where $h_\ep(z)$ is the average of $h$ over the circle of radius $\ep$ centered at $z$. 

The subsequential limits of supercritical LFPP are related to Liouville quantum gravity with \emph{matter central charge} $\ccM \in (1,25)$. LQG with $\gamma \in (0,2)$ corresponds to $\ccM = 25 - 6(2/\gamma+\gamma/2)^2 \in (-\infty,1)$. The case when $\ccM \in (1,25)$ is much less understood, even from a physics perspective. See~\cite{ghpr-central-charge,dg-supercritical-lfpp} for further discussion. 
 
In the critical case $\xi = \xi_{\op{crit}}$, there are no singular points and the subsequential limiting metrics induce the Euclidean topology. We will prove this in~\cite{dg-critical-lqg}, using the results of the present paper. This case corresponds to $\gamma$-LQG with $\gamma = 2$ or equivalently $\ccM = 1$. 

The goal of this paper is to prove, for each $\xi > 0$, up-to-constant bounds comparing $\frk a_\ep^{-1} D_h^\ep$ to any possible subsequential limit of $\{\frk a_\ep^{-1} D_h^\ep\}_{\ep > 0}$ (Theorem~\ref{thm-lqg-lfpp}). This comparison holds even at nearly microscopic scales, so is not implied by the convergence of $\frk a_\ep^{-1} D_h^\ep$. As consequences of our bounds, we will deduce estimates for the LFPP scaling constants $\{\frk a_\ep\}_{\ep > 0}$ which are new even in the subcritical case (Theorem~\ref{thm-a_eps}). We will also prove that the scaling constants $\{\frk c_r\}_{r>0}$ for the subsequential limiting metric, as defined in Axiom~\ref{item-metric-coord} of Definition~\ref{def-metric} below, can be taken to be equal to $r^{\xi Q}$ (Theorem~\ref{thm-c_r}). In the subcritical case, this was previously proven in~\cite{gm-uniqueness} as a consequence of the uniqueness of the  subsequential limit. The result is new in the critical and supercritical case, and we expect that it will help to simplify the eventual proof of the uniqueness of the subsequential limit in these cases. 

\bigskip

\noindent \textbf{Acknowledgments.} J.D.\ was partially supported by NSF grants DMS-1757479 and DMS-1953848. E.G.\ was partially supported by a Clay research fellowship. 

\subsection{Weak LQG metrics}
\label{sec-weak}

The subsequential limits of LFPP satisfy a list of axioms which define a \emph{weak LQG metric}. 
In this subsection, we will state the axiomatic definition of a weak LQG metric from~\cite{pfeffer-supercritical-lqg}. We first define the topology on the space of metrics that we will work with.

\begin{defn} \label{def-lsc} 
Let $X\subset \BB C$. 
A function $f : X \times X \rta \BB R \cup\{-\infty,+\infty\}$ is \emph{lower semicontinuous} if whenever $(z_n,w_n) \in X\times X$ with $(z_n,w_n) \rta (z,w)$, we have $f(z,w) \leq \liminf_{n\rta\infty} f(z_n,w_n)$. 
The \emph{topology on lower semicontinuous functions} is the topology whereby a sequence of such functions $\{f_n\}_{n\in\BB N}$ converges to another such function $f$ if and only if
\begin{enumerate}[(i)]
\item Whenever $(z_n,w_n) \in X\times X$ with $(z_n,w_n) \rta (z,w)$, we have $f(z,w) \leq \liminf_{n\rta\infty} f_n(z_n,w_n)$.
\item For each $(z,w)\in X\times X$, there exists a sequence $(z_n,w_n) \rta (z,w)$ such that $f_n(z_n,w_n) \rta f(z,w)$. 
\end{enumerate}
\end{defn}

It follows from~\cite[Lemma 1.5]{beer-usc} that the topology of Definition~\ref{def-lsc} is metrizable (see~\cite[Section 1.2]{dg-supercritical-lfpp}). 
Furthermore,~\cite[Theorem 1(a)]{beer-usc} shows that this metric can be taken to be separable. 

\begin{defn} \label{def-metric-properties}
Let $(X,d)$ be a metric space, with $d$ allowed to take on infinite values. 
\begin{itemize}
\item
For a curve $P : [a,b] \rta X$, the \emph{$d$-length} of $P$ is defined by 
\eqbn
\op{len}(P;d) :=  \sup_{T} \sum_{i=1}^{\# T} d(P(t_i) , P(t_{i-1})) 
\eqen
where the supremum is over all partitions $T : a= t_0 < \dots < t_{\# T} = b$ of $[a,b]$. Note that the $d$-length of a curve may be infinite.
\item
We say that $(X,d)$ is a \emph{length space} if for each $x,y\in X$ and each $\ep > 0$, there exists a curve of $d$-length at most $d(x,y) + \ep$ from $x$ to $y$. 
A curve from $x$ to $y$ of $d$-length \emph{exactly} $d(x,y)$, if such a curve exists, is called a \emph{geodesic}. 
\item
For $Y\subset X$, the \emph{internal metric of $d$ on $Y$} is defined by
\eqb \label{eqn-internal-def}
d(x,y ; Y)  := \inf_{P \subset Y} \op{len}\left(P ; d \right) ,\quad \forall x,y\in Y 
\eqe 
where the infimum is over all paths $P$ in $Y$ from $x$ to $y$. 
Note that $d(\cdot,\cdot ; Y)$ is a metric on $Y$, except that it is allowed to take infinite values.  
\item
If $X \subset \BB C$, we say that $d$ is a \emph{lower semicontinuous metric} if the function $(x,y) \rta d(x,y)$ is lower semicontinuous w.r.t.\ the Euclidean topology.  
We equip the set of lower semicontinuous metrics on $X$ with the topology on lower semicontinuous functions on $X \times X$, as in Definition~\ref{def-lsc}, and the associated Borel $\sigma$-algebra.
\end{itemize}
\end{defn}

\begin{defn} \label{def-across-around}
Let $d$ be a length metric on $\BB C$. 
For a region $A\subset\BB C$ with the topology of a Euclidean annulus, we write $d(\text{across $A$})$ for the $d$-distances between the inner and outer boundaries of $A$ and $d(\text{around $A$})$ for the infimum of the $d$-lengths of paths in $A$ which disconnect the inner and outer boundaries of $A$.  
\end{defn}

Distances around and across Euclidean annuli play a similar role to ``hard crossings" and ``easy crossings" of $2\times 1$ rectangles in percolation theory. 
One can get lower bounds for the $d$-lengths of paths in terms of the $d$-distances across the annuli that is crosses. On the other hand, one can ``string together" paths around Euclidean annuli to produce longer paths.

The following is (almost) a re-statement of~\cite[Definition 1.6]{pfeffer-supercritical-lqg}.

\begin{defn}[Weak LQG metric]
\label{def-metric}
Let $\mcl D'$ be the space of distributions (generalized functions) on $\BB C$, equipped with the usual weak topology.   
For $\xi > 0$, \emph{weak LQG metric with parameter $\xi$} is a measurable functions $h\mapsto D_h$ from $\mcl D'$ to the space of lower semicontinuous metrics on $\BB C$ with the following properties. Let $h$ be a \emph{GFF plus a continuous function} on $\BB C$: i.e., $h$ is a random distribution on $\BB C$ which can be coupled with a random continuous function $f$ in such a way that $h-f$ has the law of the whole-plane GFF. Then the associated metric $D_h$ satisfies the following axioms. 
\begin{enumerate}[I.]
\item \textbf{Length space.} Almost surely, $(\BB C,D_h)$ is a length space. \label{item-metric-length} 
\item \textbf{Locality.} Let $U\subset\BB C$ be a deterministic open set. 
The $D_h$-internal metric $D_h(\cdot,\cdot ; U)$ is a.s.\ given by a measurable function of $h|_U$.  \label{item-metric-local}
\item \textbf{Weyl scaling.} For a continuous function $f : \BB C \rta \BB R$, define
\eqb \label{eqn-metric-f}
(e^{\xi f} \cdot D_h) (z,w) := \inf_{P : z\rta w} \int_0^{\op{len}(P ; D_h)} e^{\xi f(P(t))} \,dt , \quad \forall z,w\in \BB C ,
\eqe 
where the infimum is over all $D_h$-continuous paths from $z$ to $w$ in $\BB C$ parametrized by $D_h$-length.
Then a.s.\ $ e^{\xi f} \cdot D_h = D_{h+f}$ for every continuous function $f: \BB C \rta \BB R$. \label{item-metric-f}
\item \textbf{Translation invariance.} For each deterministic point $z \in \BB C$, a.s.\ $D_{h(\cdot + z)} = D_h(\cdot+ z , \cdot+z)$.  \label{item-metric-translate}
\item \textbf{Tightness across scales.} Suppose that $h$ is a whole-plane GFF and let $\{h_r(z)\}_{r > 0, z\in\BB C}$ be its circle average process. 
There are constants $\{\frk c_r\}_{r>0}$ such that the following is true.
Let $A\subset \BB C$ be a deterministic Euclidean annulus.
In the notation of Definition~\ref{def-across-around}, the random variables  \label{item-metric-coord}
\eqbn
\frk c_r^{-1} e^{-\xi h_r(0)} D_h\left( \text{across $r A$} \right) \quad \text{and} \quad
\frk c_r^{-1} e^{-\xi h_r(0)} D_h\left( \text{around $r A$} \right)
\eqen
and the reciporicals of these random variables for $r>0$ are tight. 
\end{enumerate}
\end{defn}

For $\xi < \xi_{\op{crit}}$, it is shown in~\cite{gm-uniqueness} (as a consequence of the uniqueness of weak LQG metrics) that every weak LQG metric satisfies an exact spatial scaling property which is stronger than Axiom~\ref{item-metric-coord}. More precisely, for every $r > 0$, a.s.
\eqb \label{eqn-lqg-coord}
D_{h(r\cdot) + Q\log r}(z,w) = D_h(r z , r w) ,\quad \forall z,w\in \BB C .
\eqe 
A metric which satisfies Axioms~\ref{item-metric-length} through~\ref{item-metric-translate} together with~\eqref{eqn-lqg-coord} is called a \emph{strong LQG metric}. 

In the case when $\xi \geq \xi_{\op{crit}}$, we do not yet know that every weak LQG metric satisfies~\eqref{eqn-lqg-coord}.  
We can think of Axiom~\ref{item-metric-coord} as a substitute for~\eqref{eqn-lqg-coord}. 
It allows us to get estimates which are uniform across different Euclidean scales, even though we do not have an exact scale invariance property. 
See~\cite{lqg-metric-estimates,gm-uniqueness,pfeffer-supercritical-lqg} for further discussion of this point.

Definition~\ref{def-metric} is slightly more general than the definition of a weak LQG metric used in other works~\cite{lqg-metric-estimates,gm-uniqueness,pfeffer-supercritical-lqg}. The reason is that the aforementioned papers require a (rather weak) a priori bound for the scaling constants $\frk c_r$ from Axiom~\ref{item-metric-coord}. It follows from Theorem~\ref{thm-c_r} below that Definition~\ref{def-metric} is equivalent to the definition in~\cite{pfeffer-supercritical-lqg}, so the a priori bounds for $\frk c_r$ are unnecessary. We emphasize that our proof of Theorem~\ref{thm-c_r} does not use the results of~\cite{pfeffer-supercritical-lqg}.

\begin{remark} \label{remark-constant-multiple}
The scaling constants $\{\frk c_r\}_{r > 0}$ from Axiom~\ref{item-metric-coord} are not uniquely determined by the law of $D_h$. If $\{\wt{\frk c}_r\}_{r>0}$ is another sequence of non-negative real numbers and there is a constant $C >0$ such that $C^{-1} \frk c_r \leq \wt{\frk c}_r \leq C \frk c_r$ for each $r >0$, then Axiom~\ref{item-metric-coord} holds with $\wt{\frk c}_r$ instead of $\frk c_r$. 
Conversely, if $m_r$ denotes the median of the random variable $e^{-\xi h_r(0)} D_h(\text{around $\BB A_{r,2r}(0)$})$, then Axiom~\ref{item-metric-coord} implies that there is a constant $C > 0$, depending only on the law of $D_h$, such that $C^{-1} \frk c_r \leq m_r \leq C \frk c_r$.
In particular, any two possible choices for $\{\frk c_r\}_{r>0}$ are comparable up to constant multiplicative factors.
\end{remark}

The following theorem is proven as~\cite[Theorem 1.7]{pfeffer-supercritical-lqg}, building on the tightness result from~\cite{dg-supercritical-lfpp}.

\begin{thm}[\!\!\cite{pfeffer-supercritical-lqg}] \label{thm-lfpp-axioms}
Let $\xi > 0$. For every sequence of $\ep$'s tending to zero, there is a weak LQG metric $D$ with parameter $\xi$ and a subsequence $\{\ep_n\}_{n\in\BB N}$ for which the following is true. Let $h$ be a whole-plane GFF, or more generally a whole-plane GFF plus a bounded continuous function. Then the re-scaled LFPP metrics $\frk a_{\ep_n}^{-1} D_h^{\ep_n}$, as defined in~\eqref{eqn-gff-lfpp} and~\eqref{eqn-gff-constant}, converge in probability to $D_h$ w.r.t.\ the metric on lower semicontinuous functions on $\BB C\times \BB C$. 
\end{thm}

Theorem~\ref{thm-lfpp-axioms} implies in particular that for each $\xi > 0$, there exists a weak LQG metric with parameter $\xi$. 
In the case when $\xi < \xi_{\op{crit}}$, the convergence occurs with respect to the topology of uniform convergence on compact subsets of $\BB C\times\BB C$ and the subsequential limit has been shown to be unique~\cite{dddf-lfpp,gm-uniqueness}.

\subsection{Main results}
\label{sec-results}

Throughout this section, we fix $\xi  > 0$, we let $h$ be the whole-plane GFF, and we let $D_h$ be a weak LQG metric with parameter $\xi$. 
We recall the re-scaled LFPP metrics $\frk a_\ep^{-1} D_h^\ep$ with parameter $\xi$ from~\eqref{eqn-gff-lfpp} and~\eqref{eqn-gff-constant}. 

Our main result says that, roughly speaking, $D_h$ is bi-Lipschitz equivalent to the re-scaled LFPP metrics $\frk a_\ep^{-1} D_h^\ep$. 
We cannot say that these metrics are literally bi-Lipschitz equivalent since $D_h$ has a fractal structure whereas $\frk a_\ep^{-1} D_h^\ep$ is smooth. So, it is not possible to get an up-to-constants comparison of these metrics at Euclidean scales smaller than $\ep$. 
We get around this problem by looking at distances between Euclidean balls of radius slightly larger than $\ep$. 

\newcommand{\Cmain}{{C_0}}

\begin{thm} \label{thm-lqg-lfpp}
For each $\zeta \in (0,1) $, there exists and $\Cmain > 0$, depending only on $\zeta$ and the law of $D_h$, such that the following is true. 
Let $U\subset \BB C$ be a deterministic, connected, bounded open set and recall the notation for internal metrics on $U$ from Definition~\ref{def-metric-properties}.  
With probability tending to 1 as $\ep\rta 0$, 
\eqb \label{eqn-compare-upper'}
\frk a_\ep^{-1} D_h^\ep\left( B_{ \ep^{1-\zeta}}(z) , B_{ \ep^{1-\zeta}}(w) ; B_{ \ep^{1-\zeta}}(U) \right) 
\leq \Cmain D_h(z,w ; U)  ,\quad \forall z,w\in U
\eqe 
and
\eqb \label{eqn-compare-lower'}
D_h \left( B_{ \ep^{1-\zeta}}(z) , B_{ \ep^{1-\zeta}}(w) ; B_{ \ep^{1-\zeta}}(U) \right) 
\leq \Cmain \frk a_\ep^{-1} D_h^\ep(z,w ; U) ,\quad \forall z,w \in U . 
\eqe 
\end{thm}

Theorem~\ref{thm-lqg-lfpp} is a new result even in the subcritical case, where we already know that $\frk a_\ep^{-1} D_h^\ep \rta D_h$ in probability with respect to the topology of uniform convergence on compact subsets of $\BB C\times\BB C$. Indeed, the convergence $\frk a_\ep^{-1} D_h^\ep \rta D_h$ only allows us to estimate the ratio of $\frk a_\ep^{-1} D_h^\ep(z,w)$ and $D_h(z,w)$ when $|z-w| $ is of constant order, whereas Theorem~\ref{thm-lqg-lfpp} gives a non-trivial estimate when $|z-w|$ is as small as $3\ep^{1-\zeta}$. 

In the critical and supercritical cases, we as of yet only have subsequential limits for $\frk a_\ep^{-1} D_h^\ep$. 
In these cases, an important consequence of Theorem~\ref{thm-lqg-lfpp} is the following. 
If $\ep_k \rta 0$ is a sequence along which $\frk a_{\ep_k}^{-1} D_h^{\ep_k} \rta D_h$ in probability, then $D_h$ is a weak LQG metric~\cite{pfeffer-supercritical-lqg}. Therefore, Theorem~\ref{thm-lqg-lfpp} gives an up-to-constants comparison between $\frk a_\ep^{-1} D_h^\ep$ and $D_h$ even if $\ep$ is not part of the sequence $\{\ep_k\}$. 

We now record several estimates which are consequences of Theorem~\ref{thm-lqg-lfpp}. 
Our first estimate gives up-to-constants bounds for the scaling constants in Axiom~\ref{item-metric-coord}.
 
\newcommand{\Clqg}{{C_1}}

\begin{thm} \label{thm-c_r}
Let $\{\frk c_r\}_{r > 0}$ be the scaling constants from Axiom~\ref{item-metric-coord}. 
There is a constant $\Clqg > 0$, depending only on the law of $D_h$, such that for each $r  > 0$, 
\eqb \label{eqn-c_r}
\Clqg^{-1} r^{\xi Q} \leq \frk c_r \leq \Clqg r^{\xi Q}.
\eqe 
\end{thm}

Due to Remark~\ref{remark-constant-multiple}, Theorem~\ref{thm-c_r} is equivalent to the statement that one can take $\frk c_r = r^{\xi Q}$ in Axiom~\ref{item-metric-coord}.  
In the subcritical case $\xi < \xi_{\op{crit}}$, the exact scaling relation~\eqref{eqn-lqg-coord}, which was proven in~\cite{gm-uniqueness}, already implies that $\frk c_r = r^{\xi Q}$. However, this fact was deduced as a consequence of the uniqueness of weak LQG metrics, whereas the present paper gives a much simpler and more direct proof.

In the critical and supercritical cases, Theorem~\ref{thm-c_r} is a new result. This result will be used in~\cite{dg-critical-lqg} to show that every subsequential limit of critical LFPP induces the Euclidean topology. We also expect that it will help to simplify the proof of the uniqueness of weak LQG metrics in the critical and supercritical cases. 

Another consequence of Theorem~\ref{thm-lqg-lfpp} is the fact that any two weak LQG metrics are bi-Lipschitz equivalent.

\newcommand{\Cbilip}{{C_2}}

\begin{thm} \label{thm-bilip}
Let $D_h$ and $\wt D_h$ be two weak LQG metrics with parameter $\xi > 0$. There is a deterministic constant $\Cbilip > 1$ such that a.s.\ 
\eqb \label{eqn-bilip}
\Cbilip^{-1} D_h(z,w) \leq \wt D_h(z,w) \leq \Cbilip D_h(z,w) ,\quad\forall z,w\in\BB C .
\eqe 
\end{thm}

Previously, Theorem~\ref{thm-bilip} was established, for all $\xi > 0$, under the additional hypothesis that the constants $\frk c_r$ for the two metrics $D_h$ and $\wt D_h$ are the same (see~\cite[Theorem 1.6]{local-metrics} and~\cite[Lemma 2.23]{pfeffer-supercritical-lqg}). 
In the subcritical case, this was an important input of the proof of uniqueness in~\cite{gm-uniqueness}.

Our last estimate gives bounds for the LFPP scaling constants $\frk a_\ep$. To put our result in context, we note that previous works have shown that $\frk a_\ep = \ep^{1 - \xi Q + o_\ep(1)}$ (see~\cite[Theorem 1.5]{gm-uniqueness} for the subcritical case and~\cite[Proposition 1.9]{dg-supercritical-lfpp} for the supercritical and critical cases). In the subcritical case, the convergence of LFPP implies that $\frk a_\ep = \phi(\ep) \ep^{1-\xi Q}$, where $\phi$ is slowly varying~\cite[Corollary 1.11]{gm-uniqueness}, but one does not get any better bound than $\ep^{1-\xi Q +o_\ep(1)}$ for a fixed value of $\ep$. We improve the $\ep^{o_\ep(1)}$ error to a polylogarithmic error. 

\newcommand{\Clfpp}{{C_3}}

\begin{thm} \label{thm-a_eps}
Let $\{\frk a_\ep\}_{\ep > 0}$ be the LFPP scaling constants from~\eqref{eqn-gff-constant}. 
There are constants $\Clfpp >0$ and $b  > 0$ depending only on $\xi$ such that for each $\ep \in (0,1/2]$,  
\eqb \label{eqn-a_eps}
\Clfpp^{-1} (\log\ep^{-1})^{-b} \ep^{1 - \xi Q} \leq \frk a_\ep \leq \Clfpp (\log\ep^{-1})^b \ep^{1-\xi Q}.  
\eqe
\end{thm}

We expect that $\frk a_\ep = (c+o_\ep(1)) \ep^{1-\xi Q}$ for some $c>0$, but the techniques of the present paper are not strong enough to give this (see also Remark~\ref{remark-polylog}).

\section{Preliminaries}

\subsection{Notational conventions}
\label{sec-notation}

\noindent
We write $\BB N = \{1,2,3,\dots\}$ and $\BB N_0 = \BB N \cup \{0\}$. 
\medskip

\noindent
For $a < b$, we define the discrete interval $[a,b]_{\BB Z}:= [a,b]\cap\BB Z$. 
\medskip

\noindent
If $f  :(0,\infty) \rta \BB R$ and $g : (0,\infty) \rta (0,\infty)$, we say that $f(\ep) = O_\ep(g(\ep))$ (resp.\ $f(\ep) = o_\ep(g(\ep))$) as $\ep\rta 0$ if $f(\ep)/g(\ep)$ remains bounded (resp.\ tends to zero) as $\ep\rta 0$. We similarly define $O(\cdot)$ and $o(\cdot)$ errors as a parameter goes to infinity. 
\medskip

\noindent
Let $\{E^\ep\}_{\ep>0}$ be a one-parameter family of events. We say that $E^\ep$ occurs with \emph{polynomially high probability} as $\ep\rta 0$ if there is a $p > 0$ (independent from $\ep$ and possibly from other parameters of interest) such that  $\BB P[E^\ep] \geq 1 - O_\ep(\ep^p)$.  
\medskip

\noindent
For $z\in\BB C$ and $r>0$, we write $B_r(z)$ for the open Euclidean ball of radius $r$ centered at $z$. More generally, for $X\subset \BB C$ we write $B_r(X) = \bigcup_{z\in X} B_r(z)$. We also define the open Euclidean annulus
\eqb \label{eqn-annulus-def}
\BB A_{r_1,r_2}(z) := B_{r_2}(z) \setminus \ol{B_{r_1}(z)} ,\quad\forall 0 < r_r < r_2 < \infty .
\eqe 
\medskip
 
\subsection{Independence across concentric annuli}
\label{sec-annulus-iterate}

As is the case for many papers involving LQG distances, a key tool in our proof is the following estimate, which is a consequence of the fact that the restrictions of the GFF $h$ to disjoint concentric annuli are nearly independent. See~\cite[Lemma 3.1]{local-metrics} for a proof of a slightly more general result.

\begin{lem}[\!\!\cite{local-metrics}] \label{lem-annulus-iterate}
Fix $0 < \mu_1<\mu_2 < 1$. Let $\{r_k\}_{k\in\BB N}$ be a decreasing sequence of positive real numbers such that $r_{k+1} / r_k \leq \mu_1$ for each $k\in\BB N$ and let $\{E_{r_k} \}_{k\in\BB N}$ be events such that $E_{r_k} \in \sigma\left( (h-h_{r_k}(0)) |_{\BB A_{\mu_1 r_k , \mu_2 r_k}(0)  } \right)$ for each $k\in\BB N$ (here we use the notation for Euclidean annuli from Section~\ref{sec-notation}). 
For each $a > 0$, there exists $p = p(a,\mu_1,\mu_2) \in (0,1)$ and $c = c(a,\mu_1,\mu_2) > 0$ such that if  
\eqb \label{eqn-annulus-iterate-prob}
\BB P\left[ E_{r_k}  \right] \geq p , \quad \forall k\in\BB N  ,
\eqe 
then 
\eqb \label{eqn-annulus-iterate}
\BB P\left[ \text{$\exists k\in [1,K]_{\BB Z}$ such that $E_{r_k}$ occurs} \right] \geq 1 -  c e^{-a K} ,\quad\forall K \in \BB N. 
\eqe  
\end{lem}

\subsection{Localized approximation of LFPP}
\label{sec-localized-approx}

A somewhat annoying feature of our definition of LFPP is that the mollified process $\{h_\ep^*\}_{\ep > 0}$ does not depend locally on $h$. This is because the heat kernel $p_{\ep^2/2}(z)$ is non-zero on all of $\BB C$. In~\cite[Section 2.1]{lqg-metric-estimates}, the authors got around this difficulty by introducing a truncated version $\wh h_\ep^*$ of $h_\ep^*$ which depends locally on $h$. They then showed that LFPP defined using $\wh h_\ep^*$ instead of $h_\ep^*$ itself is a good approximate for $D_h^\ep$. 

The truncated version of LFPP used in~\cite{lqg-metric-estimates} is not quite good enough for our purposes since with the definitions used there, $\wh h_\ep^*(z)$ depends on $h|_{B_{\ep^{1/2}}(z)}$. We need a range of dependence which is smaller than $\ep^{1-\zeta}$ for every $\zeta  > 0$. So, in this subsection we will introduce a truncated version of of LFPP with a smaller range of dependence. We follow closely the exposition in~\cite[Section 2.1]{lqg-metric-estimates}, but some of our estimates are sharper. 

Fix $q > 0$.
For $\ep > 0$, let $\psi_\ep  : \BB C\rta [0,1]$ be a deterministic, smooth, radially symmetric bump function which is identically equal to 1 on $B_{ \ep (\log\ep^{-1})^q   /2  }(0)$ and vanishes outside of $B_{ \ep (\log\ep^{-1})^q    }(0)$.
We can choose $\psi_\ep$ in such a way that $(z,\ep) \mapsto \psi_\ep(z)$ is smooth. Recalling that $p_s(z )$ denotes the heat kernel, we define
\eqb \label{eqn-localized-def}
\wh h_\ep^*(z) :=  Z_\ep^{-1} \int_{\BB C} \psi_\ep(z-w) h(w) p_{\ep^2/2} (z-w) \, dw ,
\eqe
where the integral interpreted in the sense of distributional pairing and normalizing constant is given by
\eqb \label{eqn-localized-constant}
Z_\ep := \int_{\BB C}\psi_\ep(w)   p_{\ep^2/2} (w) \, dw .
\eqe
Let us now note some properties of $\wh h_\ep^*$. 

Since $\psi_\ep$ vanishes outside of $B_{\ep (\log\ep^{-1})^q   }(0)$, we have that $\wh h_\ep^*(z)$ is a.s.\ determined by $h|_{\ep B_{(\log\ep^{-1})^q   }(z)}$. 
It is easy to see that $\wh h_\ep^*$ a.s.\ admits a continuous modification (see Lemma~\ref{lem-localized-approx} below).
We henceforth assume that $\wh h_\ep^*$ is replaced by such a modification. 

If $c \in \BB R$ and we define $(\widehat{h+c})_\ep^*$ as in~\eqref{eqn-localized-def} but with $h+c$ in place of $h$, then
\eqb \label{eqn-localized-weyl}
(\widehat{h+c})_\ep^* = \wh h_\ep^*  + c .
\eqe
This is because of the normalization by $Z_\ep^{-1}$ in~\eqref{eqn-localized-def}. 

As in~\eqref{eqn-gff-sup}, we define the localized LFPP metric
\eqb \label{eqn-localized-lfpp}
\wh D_h^\ep(z,w) := \inf_{P : z\rta w} \int_0^1 e^{\xi \wh h_\ep^*(P(t))} |P'(t)| \,dt ,
\eqe
where the infimum is over all piecewise continuously differentiable paths from $z$ to $w$. 
By the definition of $\wh h_\ep^*$, 
\eqb \label{eqn-localized-property}
\text{for any open $U\subset \BB C$, the internal metric $\wh D_h^\ep(\cdot,\cdot; U)$ is a.s.\ determined by $h|_{ B_{\ep (\log\ep^{-1})^q    } (U)}$} .
\eqe

\begin{lem} \label{lem-localized-approx}
Almost surely, $(z,\ep)\mapsto \wh h_\ep^*(z)$ is continuous. 
Furthermore, for each bounded open set $U\subset \BB C$, a.s.\ 
\eqb \label{eqn-localized-approx}
\lim_{\ep\rta 0} \sup_{z\in \ol U} |h_\ep^*(z) - \wh h_\ep^*(z)|  = 0.
\eqe
In particular, a.s.\
\eqb \label{eqn-localized-lfpp-approx}
\lim_{\ep\rta 0}  \frac{\wh D_h^\ep(z,w;U)}{D_h^\ep(z,w;U)} = 1 , \quad \text{uniformly over all $z,w\in U$ with $z\not=w$}.
\eqe
\end{lem}

For the proof of Lemma~\ref{lem-localized-approx}, we will re-use the following estimate, which is~\cite[Lemma 2.2]{lqg-metric-estimates}.
 
\begin{lem} \label{lem-gff-sup} 
For each $R>0$ and $\zeta > 0$, a.s.\
\eqb \label{eqn-gff-sup}
  \sup_{z\in B_R(0)}  \sup_{r > 0}   \frac{|h_r(z)|}{  \max\{   \log(1/r) , |\log r|^{1/2+\zeta} ,  1\} }    < \infty .
\eqe
\end{lem}

\begin{proof}[Proof of Lemma~\ref{lem-localized-approx}] 
The functions $w \mapsto \psi_\ep(z-w)$ and $w\mapsto p_{\ep^2/2} (z,w)$ are each radially symmetric about $z$, i.e., they depend only on $|z-w|$. 
Using the circle average process $\{h_r\}_{r > 0}$, we may therefore write in polar coordinates
\eqb \label{eqn-localized-approx-polar}
h_\ep^*(z) = \frac{2}{ \ep^2} \int_0^\infty  r h_r(z) e^{-r^2/\ep^2} \,dr \quad \text{and} \quad
Z_\ep \wh h_\ep^*(z) = \frac{2 }{  \ep^2} \int_0^{ \ep (\log\ep^{-1})^q  } r h_r(z) \psi_\ep(r) e^{-r^2/\ep^2} \,dr .
\eqe
From this representation and the continuity of the circle average process~\cite[Proposition 3.1]{shef-kpz}, we infer that $(z,\ep) \mapsto \wh h_\ep^*(z)$ a.s.\ admits a continuous modification. 

Since $\psi_\ep \equiv 1$ on $B_{\ep (\log\ep^{-1})^q  / 2}(z)$ and $\psi_\ep$ takes values in $[0,1]$, 
\eqb \label{eqn-localized-compare-pt}
|h_\ep^*(z) - Z_\ep \wh h_\ep^*(z)| \leq  \frac{2}{\ep^2} \int_{\ep (\log\ep^{-1})^q /2}^\infty r |h_r(z)| e^{-r^2/\ep^2} \, dr .
\eqe
By Lemma~\ref{lem-gff-sup} (applied with $\zeta  =1/2$, say), a.s.\ there is a random $C = C(U)  > 0$ such that $|h_r(z) | \leq C \max\{ \log(1/r) , \log r , 1\}$ for each $z\in U$ and $r  > 0$. 
Plugging this into~\eqref{eqn-localized-compare-pt} shows that a.s.\
\eqb \label{eqn-localized-approx-Z}
\sup_{z\in U} |h_\ep^*(z) - Z_\ep \wh h_\ep^*(z)|
\leq \frac{2 C }{\ep^2} \int_{\ep (\log\ep^{-1})^q  /2}^\infty r   \max\{ \log(1/r) , \log r , 1\}     e^{-r^2/\ep^2} \, dr   ,
\eqe
which tends to zero as $\ep\rta 0$. 

To eliminate the normalizing factor $Z_\ep$ in~\eqref{eqn-localized-approx-Z}, we first note that $\int_{\BB C} p_{\ep^2/2}(w) \,dw = 1$. From this and~\eqref{eqn-localized-constant}, we have 
\allb \label{eqn-localized-Z-bound}
0 
\leq 1 - Z_\ep 
&\leq  \frac{2}{\ep^2} \int_{0}^\infty r (1 - \psi_\ep(r))  e^{-r^2/\ep^2} \, dr \notag\\ 
&\leq  \frac{2}{\ep^2} \int_{ \ep (\log\ep^{-1})^q  /2}^\infty r     e^{-r^2/\ep^2} \, dr \notag\\ 
&=  \int_{(\log\ep^{-1})^q / \sqrt 2  }^\infty u     e^{-u^2 / 2} \, du \quad \text{(substitute $r = \ep u / \sqrt 2 $)} \notag\\ 
&=      e^{- (\log\ep^{-1})^q / 4}  .
\alle
Furthermore, from Lemma~\ref{lem-gff-sup} and~\eqref{eqn-localized-approx-polar}, we get a.s.\ there is a random $A > 0$ such that 
\eqb \label{eqn-localized-sup}
\sup_{z\in U}    |Z_\ep \wh h_\ep^*(z) | \leq 3 \log \ep^{-1} + A  .
\eqe 

By the triangle inequality, 
\eqb \label{eqn-localized-tri}
\sup_{z\in U} |h_\ep^*(z) - \wh h_\ep^*(z)|
\leq \sup_{z\in U} |h_\ep^*(z) - Z_\ep \wh h_\ep^*(z)| + (1-Z_\ep) Z_\ep^{-1} \sup_{z\in U}  |Z_\ep \wh h_\ep^*(z) | .
\eqe 
By~\eqref{eqn-localized-approx-Z}, the first term on the right side of~\eqref{eqn-localized-tri} tends to zero a.s.\ as $\ep\rta 0$. By~\eqref{eqn-localized-Z-bound} and~\eqref{eqn-localized-sup}, the second term tends to zero a.s.\ as $\ep\rta 0$ as well. We thus obtain~\eqref{eqn-localized-approx}. 
The relation~\eqref{eqn-localized-lfpp-approx} is immediate from~\eqref{eqn-localized-approx} and the definitions of $D_h^\ep$ and $\wh D_h^\ep$. 
\end{proof}

\section{Proofs}
\label{sec-proofs}

Throughout this section, we fix $\xi > 0$, we let $h$ be a whole-plane GFF, and we let $D_h$ be a weak LQG metric as in Definition~\ref{def-metric}. 
We will prove Theorem~\ref{thm-lqg-lfpp} in Sections~\ref{sec-event} through~\ref{sec-lqg-lfpp}, then deduce our other main theorems from Theorem~\ref{thm-lqg-lfpp} in Section~\ref{sec-scaling}.

\subsection{Outline of the proof}
\label{sec-outline}

The idea of the proof of Theorem~\ref{thm-lqg-lfpp} is as follows. In Section~\ref{sec-event}, we use a basic scaling calculation together with the tightness of LFPP~\cite{dddf-lfpp,dg-supercritical-lfpp} and the tightness across scales condition for $D_h$ (Axiom~\ref{item-metric-coord}) to get the following. There is a constant $C>0$ such that for each $r\in [\ep,1]$ and each $z\in\BB C$, it holds with high probability that 
\eqb \label{eqn-outline-event}
\frk a_\ep^{-1} D_h^\ep\left(\text{around $\BB A_{2r,3r}(z)$} \right) \leq  C  \frac{r \frk a_{\ep/r}}{\frk c_r \frk a_\ep} D_h(\text{across $\BB A_{r,2r}(z)$})  
\eqe 
where we use the notation for distances across and around annuli from Definition~\ref{def-across-around}. Moreover, one has an analogous inequality with the roles of $D_h$ and $\frk a_\ep^{-1} D_h^\ep$ interchanged. Actually, for technical reasons we will mostly work with the localized LFPP metric $\wh D_h^\ep$ from Section~\ref{sec-localized-approx} instead of $D_h^\ep$ itself. 

In Section~\ref{sec-scales}, we restrict attention to values of $r$ in the set
\eqb \label{eqn-outline-radius-set}
\mcl R^\ep := \left\{ 10^{-j} \ep^{1-\zeta} : j = 0,\dots, \lfloor \frac{\zeta}{2} \log_{10}\ep^{-1} \rfloor -11 \right\}  \subset \left[ \ep ,\ep^{1-\zeta}  \right]  .
\eqe 
We will use a multi-scale argument based on Lemma~\ref{lem-annulus-iterate} to say that if $\wt{\mcl R}^\ep \subset \mcl R^\ep$ is a subset with $\#\wt{\mcl R}^\ep \geq \#\mcl R^\ep / 100$, then the following is true. For each open set $U\subset \BB C$, it holds with high probability that we can cover $U$ by balls $B_r(z)$ for $z\in U$ and $r\in \wt{\mcl R}^\ep$ for which the event in~\eqref{eqn-outline-event} occurs. By stringing together paths around annuli of the form $\BB A_{2r,3r}(z)$ whose $D_h^\ep$-lengths are under control, we will then deduce that with high probability, 
\eqb \label{eqn-outline-compare}
\frk a_\ep^{-1} D_h^\ep\left( B_{ \ep^{1-\zeta}}(z) , B_{  \ep^{1-\zeta}}(w) ; B_{ \ep^{1-\zeta}}(U) \right) 
\leq C \left(\max_{r \in \wt{\mcl R}^\ep } \frac{r \frk a_{\ep/r}}{\frk c_r \frk a_\ep} \right) D_h(z,w ; U) ,\quad \forall z,w\in U  
\eqe 
for a possibly larger constant $C$. One also has an analogous inequality with the roles of $D_h$ and $\frk a_\ep^{-1} D_h^\ep$ interchanged. The argument leading to~\eqref{eqn-outline-compare} is similar to the proof of the bi-Lipschitz equivalence of two metrics coupled to the GFF in~\cite[Section 4]{local-metrics}. 

The inequalities~\eqref{eqn-outline-compare} would imply Theorem~\ref{thm-lqg-lfpp} if we had up-to-constants bounds for the ratios $\frac{r \frk a_{\ep/r}}{\frk c_r \frk a_\ep}$ for $r\in \wt{\mcl R}^\ep$. 
In Section~\ref{sec-lqg-lfpp}, we obtain such up-to-constants bounds via the following bootstrap argument based on~\eqref{eqn-outline-compare}. If we fix distinct points $z,w\in U$, then $\frk a_\ep^{-1} D_h^\ep(z,w)$ and $D_h(z,w)$ are each typically of constant order. The inequality~\eqref{eqn-outline-compare} therefore implies that $\max_{r \in \wt{\mcl R}^\ep } \frac{r \frk a_{\ep/r}}{\frk c_r \frk a_\ep} \geq A^{-1}$ for some constant $A$ which does not depend on $r$ or $\ep$. Using the analog of~\eqref{eqn-outline-compare} with the roles of $D_h$ and $\frk a_\ep^{-1} D_h^\ep$ interchanged and possibly increasing $A$, we also get that $\min_{r \in \wt{\mcl R}^\ep } \frac{r \frk a_{\ep/r}}{\frk c_r \frk a_\ep} \leq A $. 
Since $\wt{\mcl R}^\ep$ is an arbitrary subset of $\mcl R^\ep$ of cardinality at least $\#\mcl R^\ep / 100$, this implies that all but $( 2/100) \#\mcl R^\ep$ values of $r\in \mcl R^\ep$ are ``good" in the sense that $\frac{r \frk a_{\ep/r}}{\frk c_r \frk a_\ep} \in [A^{-1} , A]$ . We then obtain Theorem~\ref{thm-lqg-lfpp} by applying~\eqref{eqn-outline-compare} (and its analog with $\frk a_\ep^{-1} D_h^\ep$ and $D_h$ interchanged) with $\wt{\mcl R}^\ep$ equal to the set of ``good" values of $r\in \mcl R^\ep$. 

\subsection{Event for LFPP and LQG distances}
\label{sec-event}

Fix $q>0$ and for $\ep > 0$, let $\wh D_h^\ep$ be the localized LFPP metric of Section~\ref{sec-localized-approx}. 
For $z\in \BB C$, $\ep \in (0,1)$, $r\in [\ep,1]$, and $C>0$ let $  E_r^\ep(z;C)$ be the event that the following is true:
\allb \label{eqn-compare-event}
D_h\left( \text{across $\BB A_{r,2r}(z)$} \right) &\geq C^{-1} \frk c_r e^{\xi h_r(z)} \notag \\
D_h\left( \text{around $\BB A_{2r,3r}(z)$} \right) &\leq C \frk c_r e^{\xi h_r(z)} \notag \\
\frk a_\ep^{-1} \wh D_h^\ep\left( \text{across $\BB A_{r,2r}(z)$} \right) &\geq C^{-1} \frac{r  \frk a_{\ep/r}}{\frk a_\ep}   e^{\xi h_r(z)} \notag \\
\frk a_\ep^{-1} \wh D_h^\ep \left( \text{around $\BB A_{2r,3r}(z)$} \right) &\leq C \frac{r  \frk a_{\ep/r}}{\frk a_\ep} e^{\xi h_r(z)}  .
\alle

We will eventually apply the independence of the GFF across disjoint concentric annuli (Lemma~\ref{lem-annulus-iterate}) and a union bound to show that if $C$ is large enough, then with high probability there are many points $z$ and radii $r$ for which $E_r(z)$ occurs. Let us now explain why the event $E_r(z;C)$ fits into the framework of Lemma~\ref{lem-annulus-iterate}. 

If $r > 2\ep (\log\ep^{-1})^{q}  $, then by the locality of $D_h$ (Axiom~\ref{item-metric-local}) and~\eqref{eqn-localized-property}, the event $E_r^\ep(z;C)$ is a.s.\ determined by the restriction of $h$ to the annulus $\BB A_{r/2,4r}(z)$. This locality property for $E_r^\ep(z;C)$ is the reason why we use $\wh D_h^\ep$ instead of $D_h^\ep$ in the definition. Since $D_h  $ and $\wh D_h^\ep$ transform in the same way when we add a constant to $h$ (see Axiom~\ref{item-metric-f} and~\eqref{eqn-localized-weyl}), we see that in fact $E_r^\ep(z;C)$ is a.s.\ determined by $h|_{\BB A_{r/2,4r}(z)}$ viewed modulo additive constant. 

Using tightness across scales (Axiom~\ref{item-metric-coord}) along with the tightness and scaling properties of LFPP, we can show that $E_r(z;C)$ occurs with high probability when $C$ is large.

\begin{lem} \label{lem-compare-event-prob}
For each $p\in (0,1)$, there exists $C   > 0$, depending on $p,\xi,$ and the law of $D_h$, such that
\eqb \label{eqn-compare-event-prob}
\BB P\left[ E_r^\ep(z;C) \right] \geq p ,\quad \forall z\in\BB C,\quad \forall 0 < \ep \leq r \leq 1 . 
\eqe 
\end{lem}
\begin{proof}
Since the event $E_r^\ep(z;C)$ is determined by $h$, viewed modulo additive constant, the law of $h$ is translation invariant modulo additive constant, and $D_h$ and $\wh D_h^\ep$ depend on $h$ in a translation invariant way, we see that $\BB P[E_r^\ep(z;C)]$ does not depend on $z$. So, we can restrict attention to the case when $z = 0$.
Let $\wt E_r^\ep(C)$ be define in the same manner as the event $E_r^\ep(0;C)$ of~\eqref{eqn-compare-event}, but with the ordinary LFPP metric $D_h^\ep$ instead of the truncated LFPP metric $\wh D_h^\ep$. The reason we want to look at $D_h^\ep$ is the scaling property in~\eqref{eqn-use-lfpp-scale} below.
 
Due to the uniform comparison between the fields $\wh h_\ep^*$ and $h_\ep^*$ involved in the definitions of $\wh D_h^\ep$ and  $D_h^\ep$, (Lemma~\ref{lem-localized-approx}), it suffices to show that for each $p\in (0,1)$, there exists $C = C(p,\xi) >0 $ such that
\eqb \label{eqn-compare-event-show}
\BB P\left[ \wt E_r^\ep(C)  \right] \geq p ,\quad  \forall 0 < \ep \leq r \leq 1 .
\eqe 

To prove~\eqref{eqn-compare-event-show}, we first use tightness across scales (Axiom~\ref{item-metric-coord}) to find $C    > 0$ as in the lemma statement such that for each $r\in (0,1]$, it holds with probability at least $(1+p)/2$ that
\eqb \label{eqn-compare-lqg}
D_h\left( \text{across $\BB A_{r,2r}(0)$} \right)  \geq C^{-1} \frk c_r e^{\xi h_r(0)} \quad \text{and} \quad
D_h\left( \text{around $\BB A_{2r,3r}(0)$} \right)  \leq C \frk c_r e^{\xi h_r(0)}  .
\eqe 

To lower-bound the probabilities of the conditions for $\frk a_\ep^{-1} D_h^\ep$ in the definition of $\wt E_r^\ep(C)$, we write $h^r = h(r\cdot) - h_r(0)$, which has the same law of $h$. A simple scaling calculation using the definitions of $h_\ep^*$ and $D_h^\ep$ shows that
\eqb  \label{eqn-use-lfpp-scale}
\frk a_\ep^{-1}  D_h^\ep(r u , r v) 
= \frac{r \frk a_{\ep/r}}{\frk a_\ep} e^{ \xi h_r(0)}   \times   \frk a_{\ep/r}^{-1} D_{h^r}^{\ep /r} (u,v)  ,
 \quad \forall u , v \in \BB C ; 
\eqe  
see~\cite[Lemma 2.6]{lqg-metric-estimates}.
Since $h^r \eqD h$, we can use the tightness of crossing distances for LFPP~\cite[Proposition 4.1]{dg-supercritical-lfpp} to get that, after possibly increasing $C$, for each $0 < \ep \leq r \leq 1$, it holds with probability at least $(1+p)/2$ that
\eqb \label{eqn-compare-lfpp0}
\frk a_{\ep/r}^{-1} D_{h^r}^{\ep /r}\left( \text{across $\BB A_{1,2}(0)$} \right) \geq C^{-1}  \quad \text{and} \quad
\frk a_{\ep/r}^{-1} D_{h^r}^{\ep /r}\left( \text{around $\BB A_{2,3}(0)$} \right) \leq C .
\eqe 
By~\eqref{eqn-use-lfpp-scale} and~\eqref{eqn-compare-lfpp0}, for each $0 < \ep \leq r \leq 1$, it holds with probability at least $(1+p)/2$ that
\eqb \label{eqn-compare-lfpp}
\frk a_\ep^{-1}  D_h^\ep\left( \text{across $\BB A_{r,2r}(0)$} \right) \geq C^{-1} \frac{r \frk a_{\ep/r}}{\frk a_\ep} e^{ \xi h_r(0)}    \quad \text{and} \quad
\frk a_\ep^{-1}  D_h^\ep\left( \text{around $\BB A_{2r,3r}(0)$} \right) \leq C \frac{r \frk a_{\ep/r}}{\frk a_\ep} e^{ \xi h_r(0)}  .
\eqe

Combining~\eqref{eqn-compare-lqg} with~\eqref{eqn-compare-lfpp} gives~\eqref{eqn-compare-event-show}. 
\end{proof}

\subsection{Comparing LFPP and LQG distances using distances at smaller scales}
\label{sec-scales}
  
Let us now define the set of radii which we will consider when we apply Lemma~\ref{lem-annulus-iterate}. 
Fix $\zeta \in (0,1)$ and for $\ep \in (0,1)$, let 
\eqbn
N^\ep := \lfloor \frac{\zeta}{2} \log_{10}\ep^{-1} \rfloor -10 .
\eqen
Let
\eqb \label{eqn-radius-set}
\mcl R^\ep := \left\{ 10^{-j} \ep^{1-\zeta} : j = 0,\dots, N^\ep -1 \right\}  \subset \left[  \ep^{1-\zeta/2} , \frac{1}{100} \ep^{1-\zeta}  \right]  .
\eqe 
We note that $ \#\mcl R^\ep =  N^\ep$ and for any $r,r'\in\mcl R^\ep$ with $r < r'$, we have $r'/r \geq 10$.  

The following lemma tells us that if $\wt{\mcl R}^\ep \subset\mcl R^\ep$ is a large enough subset, then with high probability we can compare $D_h$-distances and $\frk a_\ep^{-1} \wh D_h^\ep$-distances at scales larger than $\ep^{1-\zeta}$, up to a factor depending on the ratios $ r \frk a_{\ep/r} / ( \frk c_r \frk a_\ep ) $ for $r \in \wt{\mcl R}^\ep$. 

\newcommand{\Csub}{{C_4}}

\begin{lem} \label{lem-compare-subset}
There exists $\Csub > 0$, depending on $\zeta$ and the law of $D_h$, such that the following is true.
Let $U\subset \BB C$ be a deterministic, connected, bounded open set.
Also let $\ep\in (0,1)$ and let $\wt{\mcl R}^\ep \subset \mcl R^\ep$ be a deterministic subset with $\# \wt{\mcl R}^\ep \geq N^\ep/100$. 
It holds with polynomially high probability as $\ep\rta 0$, at a rate depending only on $U,\zeta$, and the law of $D_h$, that the following is true. 
For each $z,w\in U$, we have
\eqb \label{eqn-compare-upper}
\frk a_\ep^{-1} \wh D_h^\ep\left( B_{ \ep^{1-\zeta}}(z) , B_{  \ep^{1-\zeta}}(w) ; B_{ \ep^{1-\zeta}}(U) \right) 
\leq \Csub \left(\max_{r \in \wt{\mcl R}^\ep } \frac{r \frk a_{\ep/r}}{\frk c_r \frk a_\ep} \right) D_h(z,w ; U) 
\eqe 
and
\eqb \label{eqn-compare-lower}
D_h \left( B_{ \ep^{1-\zeta}}(z) , B_{ \ep^{1-\zeta}}(w) ; B_{ \ep^{1-\zeta}}(U) \right) 
\leq \Csub \left(\min_{r \in \wt{\mcl R}^\ep } \frac{r \frk a_{\ep/r}}{\frk c_r \frk a_\ep} \right)^{-1}  \frk a_\ep^{-1} \wh D_h^\ep(z,w ; U) .
\eqe 
\end{lem}

The statement of Lemma~\ref{lem-compare-subset} is similar to the statement of Theorem~\ref{thm-lqg-lfpp}, except that in Lemma~\ref{lem-compare-subset} our estimates have an extra factor which depends on the ratios $  r \frk a_{\ep/r}  (\frk c_r \frk a_\ep)$ for $r\in\wt{\mcl R}^\ep$. In Section~\ref{sec-lqg-lfpp}, we will deduce Theorem~\ref{thm-lqg-lfpp} from Lemma~\ref{lem-compare-subset} by finding a choice of $\wt{\mcl R}^\ep$ for which these ratios are of constant order.

The proof of Lemma~\ref{lem-compare-subset} is similar to the proof in~\cite[Section 4]{local-metrics} that two metrics coupled with the GFF which satisfy certain conditions are bi-Lipschitz equivalent. We will first use Lemma~\ref{lem-annulus-iterate} to find lots of points $z$ and radii $r\in \wt{\mcl R}^\ep$ for which $E_r^\ep(z;C)$ occurs (Lemma~\ref{lem-compare-iterate}). By the definition~\eqref{eqn-compare-event} of $E_r^\ep(z;C)$, for each such $z$ and $r$,
\eqb \label{eqn-compare-subset-ineq}
\frk a_\ep^{-1} \wh D_h^\ep \left( \text{around $\BB A_{2r,3r}(z)$} \right) \leq C^2 \frac{r  \frk a_{\ep/r}}{\frk c_r \frk a_\ep} D_h\left( \text{across $\BB A_{r,2r}(z)$} \right)   
\eqe 
and a similar inequality holds with the roles of $\frk a_\ep^{-1 }\wh D_h^\ep$ and $D_h$ interchanged. To prove~\eqref{eqn-compare-upper}, we will consider a $D_h$-geodesic $P$. We will then string together paths around annuli of the form $\BB A_{2r,3r}(z)$ which intersect $P$ in order to produce a path with approximately the same endpoints as $P$. Using ~\eqref{eqn-compare-subset-ineq}, we can arrange that the $\frk a_\ep^{-1} \wh D_h^\ep$-length of this new path is bounded above in terms of the $D_h$-length of $P$. The bound~\eqref{eqn-compare-lower} is proven via a similar argument with the roles of $\frk a_\ep^{-1 }\wh D_h^\ep$ and $D_h$ interchanged.

\newcommand{\Citerate}{{C_5}}

\begin{lem} \label{lem-compare-iterate}
Assume the setup of Lemma~\ref{lem-compare-subset}. 
There exists $\Citerate  > 0$, depending only on $\zeta$ and the law of $D_h$, such that with polynomially high probability as $\ep\rta 0$, at a rate depending only on $U,\zeta$, and the law of $D_h$, that the following is true. 
For each $u \in \left( \frac{\ep}{100} \BB Z^2 \right) \cap B_1(U)$, there exists $r \in \wt{\mcl R}^\ep$ such that the event $E_r^\ep(u;\Citerate)$ occurs. 
\end{lem}
\begin{proof}
We have $\#\wt{\mcl R}^\ep \geq N^\ep/100 \geq \lfloor \frac{\zeta}{200} \log_{10} \ep^{-1} \rfloor$. Moreover, by~\eqref{eqn-radius-set}, if we list the elements of $\wt{\mcl R}^\ep$ in numerical order, then the ratio of any two consecutive elements is at least $10$. For each $r\in\wt{\mcl R}^\ep$, we have $r\geq \ep^{1-\zeta/2} \geq \ep (\log\ep^{-1})^q  $, so as explained just after~\eqref{eqn-compare-event}, the event $E_r^\ep(u;C)$ is a.s.\ determined by $h|_{\BB A_{r/2 , 4r}(u)}$, viewed modulo additive constants. By Lemma~\ref{lem-compare-event-prob}, for any $p\in (0,1)$ we can choose $C = C(p,\xi)>0$ such that $\BB P\left[ E_r^\ep(u;C) \right] \geq p$ for each $r\in\wt{\mcl R}^\ep$ and each $u\in\BB C$. From this and Lemma~\ref{lem-annulus-iterate} (applied with $K = \#\wt{\mcl R}^\ep$, the radii $r_k$ equal to the elements of $\wt{\mcl R}^\ep$, and $a $ equal to a large constant times $1/\zeta$), we find that there exists $\Citerate    > 0$ as in the lemma statement such that for each $u\in\BB C$,
\eqbn
\BB P\left[ \text{$\exists r\in \wt{\mcl R}^\ep$ such that $E_r^\ep(u;\Citerate)$ occurs} \right] = 1 - O_\ep(\ep^{100}) .
\eqen
We now conclude via a union bound over $O_\ep(\ep^2)$ elements of $\left( \frac{\ep}{100} \BB Z^2 \right) \cap B_1(U)$. 
\end{proof}

We now turn our attention to the proof of Lemma~\ref{lem-compare-subset}.
Let $\Citerate  > 0$ be as in Lemma~\ref{lem-compare-iterate}. Throughout the proof, we assume that the event of Lemma~\ref{lem-compare-iterate} occurs, which happens with polynomially high probability as $\ep\rta 0$. 
We will show (via a purely deterministic argument) that~\eqref{eqn-compare-upper} holds. The proof of~\eqref{eqn-compare-lower} is similar, with the roles of $\frk a_\ep^{-1} \wh D_h^\ep$ and $D_h$ interchanged.

To this end, fix distinct points $z,w\in U$ and let $P : [0,T] \rta U$ be a path in $U$ from $z$ to $w$ with $D_h$-length at most $2D_h(z,w;U)$. 
We assume that $P$ is parameterized by its $D_h$-length. 
We will build a path from $B_{ \ep^{1-\zeta}}(z)$ to $B_{ \ep^{1-\zeta}}(w)$ which approximates $P$ and whose $\frk a_\ep^{-1} \wh D_h^\ep$-length is bounded above.

\begin{figure}[ht!]
\begin{center}
\includegraphics[scale=1]{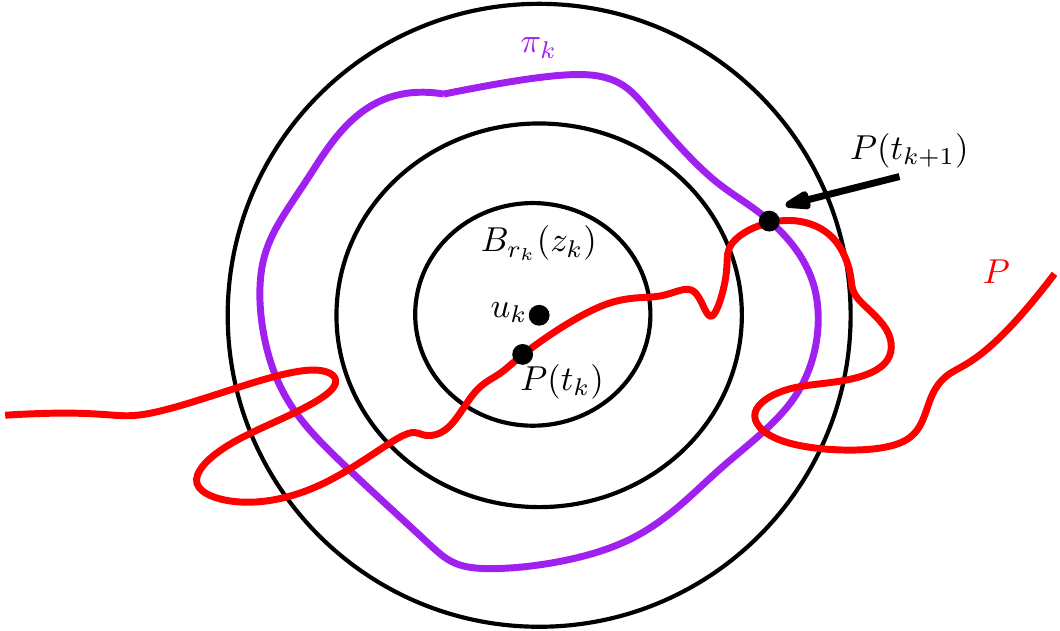} 
\caption{\label{fig-compare-subset} Illustration of the objects involved in one step of the iterative construction of the times $t_k$. The red path is a segment of $P$ and the two annuli in the figure are $\BB A_{r_k,2r_k}(u_k)$ and $\BB A_{2r_k,3r_k}(u_k)$. 
}
\end{center}
\end{figure}

To do this, we first inductively define a sequence of times $\{t_k\}_{k\in\BB N_0} \subset [0,T]$.
See Figure~\ref{fig-compare-subset} for an illustration of the definitions.
Let $t_0  = 0$. Inductively, assume that $k \in \BB N_0$ and $t_k$ has been defined. 
If $t_k = T$, we set $t_{k+1} = T$. 
Otherwise, we choose $u_k \in \left( \frac{\ep}{100} \BB Z^2 \right) \cap B_1(U)$ such that $P(t_k) \in B_\ep(u_k)$.
Since we are assuming that the event of Lemma~\ref{lem-compare-iterate} occurs, we can choose $r_k \in \wt{\mcl R}^\ep$ such that $E_{r_k}^\ep(u_k ; \Citerate)$ occurs.
By the definition~\eqref{eqn-compare-event} of $E_{r_k}^\ep(u_k ; \Citerate)$, there exists a path $\pi_k \subset \BB A_{2 r_k , 3 r_k}(u_k)$ which disconnects the inner and outer boundaries of this annulus such that
\eqb \label{eqn-bilip-path}
\left(\text{$\frk a_\ep^{-1} \wh D_h^\ep$-length of $\pi_k$} \right) \leq 2 \Citerate \frac{r_k  \frk a_{\ep/r_k}}{\frk a_\ep} e^{\xi h_{r_k}(u_k)} .
\eqe 
We can take $\pi_k$ to be a Jordan curve (i.e., a homeomorphic image of the circle). 
Let $t_{k+1}$ be the first time after $t_k$ at which the path $P$ hits $\pi_k$, or $t_{k+1} = T$ if no such time exists. 
Let
\eqbn
K := \max\{k\in\BB N: t_k < T\} .
\eqen
 
By the definition~\eqref{eqn-radius-set} of $\mcl R^\ep$, we have $r_k \geq \ep$ for each $k$, so by our choice of $u_k$ we have $P(t_k) \in B_{r_k}(u_k)$. Since $\pi_k \subset \BB A_{2 r_k , 3 r_k}(u_k)$, we see that if $k+1 \leq K$, then $P$ must cross between the inner and outer boundaries of $\BB A_{r_k , 2r_k}(u_k)$ between times $t_k$ and $t_{k+1}$. Since $P$ is parametrized by $D_h$-length and by~\eqref{eqn-compare-event},
\eqb \label{eqn-bilip-increment}
t_{k+1} - t_k \geq D_h\left(\text{across $\BB A_{r_k , 2r_k}(u_k)$} \right) \geq \Citerate^{-1} \frk c_{r_k} e^{\xi h_{r_k}(u_k)} , \quad \forall k  \leq  K -1 .
\eqe 
Therefore, 
\allb \label{eqn-bilip-sum}
\sum_{k=0}^{K-1} \left(\text{$\frk a_\ep^{-1} \wh D_h^\ep$-length of $\pi_k$} \right) 
&\leq \sum_{k=0}^{K-1} 2 \Citerate \frac{r_k  \frk a_{\ep/r_k}}{\frk a_\ep} e^{\xi h_{r_k}(u_k)} \quad \text{(by~\eqref{eqn-bilip-path})}\\
&\leq \sum_{k=0}^{K-1} 2 \Citerate^2 \frac{r_k  \frk a_{\ep/r_k}}{ c_{r_k} \frk a_\ep} (t_{k+1} - t_k)   \quad \text{(by~\eqref{eqn-bilip-increment})}\\
&\leq  2 \Citerate^2 \left( \sup_{r\in \wt{\mcl R}^\ep}  \frac{r  \frk a_{\ep/r }}{ c_{r } \frk a_\ep} \right) \sum_{k=0}^{K-1} (t_{k+1} - t_k)    \\
&\leq 4  \Citerate^2 \left( \sup_{r\in \wt{\mcl R}^\ep}  \frac{r  \frk a_{\ep/r }}{ c_{r } \frk a_\ep} \right) D_h(z,w;U) \quad \text{(by our choice of $P$)}   .
\alle

By definition, each of the paths $\pi_k$ for $k=0,\dots,K$ intersects $P$, which is contained in $U$, and has Euclidean diameter at most $6r_k \leq  \ep^{1-\zeta}$. 
Therefore,
\eqb \label{eqn-bilip-union}
\bigcup_{k=0}^{K-1} \pi_k \subset B_{ \ep^{1-\zeta}}(U) .
\eqe

In light of~\eqref{eqn-bilip-increment} and~\eqref{eqn-bilip-union}, to conclude the proof of~\eqref{eqn-compare-upper} (with $4\Citerate^2$ instead of $\Csub$) it remains to prove the following topological lemma.

\begin{lem} \label{lem-connected}
In the notation above, the union of the paths $\pi_k$ for $k=0,\dots,K-1$ contains a path from $B_{ \ep^{1-\zeta}}(z)$ to $B_{ \ep^{1-\zeta}}(w)$. 
\end{lem}

Indeed, once Lemma~\ref{lem-connected} is established,~\eqref{eqn-bilip-sum} implies that the $\frk a_\ep^{-1} \wh D_h^\ep$-length of the path from the lemma is at most $2  \Citerate^2 \left( \sup_{r\in \wt{\mcl R}^\ep}  \frac{r  \frk a_{\ep/r }}{ c_{r } \frk a_\ep} \right) D_h(z,w;U)$ and~\eqref{eqn-bilip-union} implies that the path from the lemma is contained in $B_{ \ep^{1-\zeta}}(U)$. Hence~\eqref{eqn-compare-upper} holds with $\Csub = 4\Citerate^2$. 

\begin{proof}[Proof of Lemma~\ref{lem-connected}]
For $k=0,\dots,K-1$, let $V_k$ be the open region which is disconnected from $\infty$ by the path $\pi_k$. Since $\pi_k$ is a Jordan curve, we have that $V_k$ is bounded and $\bdy V_k = \pi_k$. 
By construction, $P\subset \bigcup_{k=0}^{K-1} V_k$. 
Furthermore, the Euclidean diameter of each $V_k$ is at most $6r_k \leq  \ep^{1-\zeta}$. 
Let $\mcl K \subset [0,K-1]_{\BB Z}$ be a subset which is minimal in the sense that $ P \subset \bigcup_{k \in \mcl K} V_k$ and $P$ is not covered by any proper subcollection of the sets $V_k$ for $k\in\mcl K$. 

Since $ P $ is connected, it follows that $\bigcup_{k \in \mcl K} V_k$ is connected. Indeed, if this set had two proper disjoint open subsets, then each would have to intersect $ P $ (by minimality) which would contradict the connectedness of $ P $. 
Furthermore, by minimality, none of the sets $V_k$ for $k\in\mcl K$ is properly contained in a set of the form $V_{k'}$ for $k'\in\mcl K$. 

We claim that $\bigcup_{k\in\mcl K} \pi_k$ is connected. Indeed, if this were not the case then we could partition $\mcl K = \mcl K_1 \cup \mcl K_2$ such that $\mcl K_1$ and $\mcl K_2$ are non-empty and $\bigcup_{k\in\mcl K_1} \pi_k$ and $\bigcup_{k\in\mcl K_2} \pi_k$ are disjoint.  
By the minimality of $\mcl K$, it cannot be the case that any of the sets $V_{k'}$ for $k' \in\mcl K_2$ is contained in $\bigcup_{k\in\mcl K_1} V_k$. 
Furthermore, since $\bigcup_{k\in\mcl K_1} \pi_k$ and $\bigcup_{k\in\mcl K_2} \pi_k$ are disjoint, it cannot be the case that any set of the form $V_{k'}$ for $k' \in \mcl K_2$ intersects both $\bigcup_{k\in\mcl K_1} V_k$ and $\BB C\setminus \bigcup_{k\in\mcl K_1} V_k$: indeed otherwise $V_{k'}$ would have to intersect $\bdy V_k = \pi_k$ for some $k\in \mcl K_1$, which would mean that either $V_{k'} \supset V_k$ or $\pi_{k'}$ intersects $\pi_k$. The first case is impossible by the preceding paragraph and the second case is impossible by our choice of $\mcl K_1$ and $\mcl K_2$. 
Hence $V_{k'}$ must be entirely contained in $\BB C\setminus \bigcup_{k\in\mcl K_1} V_k$.
Therefore, $\bigcup_{k\in\mcl K_1} V_k$ and $\bigcup_{k\in\mcl K_2} V_k$ are disjoint.  
This contradicts the connectedness of $\bigcup_{k \in \mcl K} V_k$, and therefore gives our claim. 

Since $\bigcup_{k \in\mcl K} V_k$ contains $ P $, each of the sets $V_k$ has Euclidean diameter at most $ \ep^{1-\zeta}$, and $\bigcup_{k\in\mcl K} \pi_k$ is connected, it follows that $\bigcup_{k\in\mcl K} \pi_k$ contains a path from $B_{ \ep^{1-\zeta}}(z)$ to $B_{ \ep^{1-\zeta}}(w)$, as required (recall that $P(0) = z$ and $P(T) =w$). 
\end{proof}

\subsection{Up-to-constants comparison of LFPP and LQG distances}
\label{sec-lqg-lfpp}
 
In order to deduce Theorem~\ref{thm-lqg-lfpp} from Lemma~\ref{lem-compare-subset}, we need to produce a large subset of $\mcl R^\ep$ such that the ratios $ r \frk a_{\ep/r} / (\frk c_r \frk a_\ep ) $ for $r\in\mcl R^\ep$ are of constant order.
The existence of such a subset turns out to be a consequence of Lemma~\ref{lem-compare-subset}. 
Indeed, for fixed distinct points $z,w\in \BB C$ we know a priori that $\frk a_\ep^{-1} \wh D_h^\ep(z,w)$ and $D_h(z,w)$ are each typically of constant order. 
If there were a large number of scales $r\in\mcl R^\ep$ for which $  r \frk a_{\ep/r} / (\frk c_r \frk a_\ep) $ is very small, then Lemma~\ref{lem-compare-subset} would imply that $\frk a_\ep^{-1} \wh D_h^\ep(z,w)$ is typically much smaller than a small constant times $D_h(z,w)$, which is impossible. Similarly, there cannot be too many values of $r\in \mcl R^\ep$ for which $  r \frk a_{\ep/r}  / (\frk c_r \frk a_\ep )$ is very large. Hence this ratio must be of constant order for ``most" $r\in\mcl R^\ep$.  
Let us now make this reasoning precise.

\newcommand{\Cratio}{{C_6}}

\begin{lem} \label{lem-good-ratios}
There exists $\Cratio   > 1$, depending only on $\zeta$ and the law of $D_h$, such that for each $\ep \in (0,1)$, there are at least $N^\ep/2$ values of $r\in\mcl R^\ep$ such that 
\eqb \label{eqn-good-ratios}
\Cratio^{-1} \leq \frac{r \frk a_{\ep/r}}{\frk c_r \frk a_\ep} \leq \Cratio .
\eqe 
\end{lem}
\begin{proof}
For any $\ep_0 > 0$, the scaling constants $\frk a_\ep$ for $\ep\in [\ep_0,1]$ are bounded above and below by constants depending only on $\ep_0$ and $\xi$ and the constants $\frk c_r$ for $r \in [\ep_0 , 1]$ are bounded above and below by constants depending only on $\ep_0$ and the law of $D_h$. Hence, we can choose $\Cratio  > 1$ depending only on $\ep_0,\zeta$, and the law of $D_h$ such that~\eqref{eqn-good-ratios} holds for all $\ep \in [\ep_0,1]$ and all $r\in [\ep,\ep^{1-\zeta}]$. 
Therefore, it suffices to find $\Cratio > 1$ as in the lemma statement such that the lemma statement holds for each small enough $\ep >0$ (how small depends on $\zeta $ and the law of $D_h$).  

For $T > 1$, let $\wt{\mcl R}_{T,+}^\ep$ (resp.\ $\wt{\mcl R}_{T,-}^\ep$) be the set of $r\in\mcl R^\ep$ such that $ r \frk a_{\ep/r} / (\frk c_r \frk a_\ep)  > T$ (resp.\ $ r \frk a_{\ep/r} / (\frk c_r \frk a_\ep ) < T^{-1}$). If the lemma statement does not hold with $\Cratio = T$, then either $\# \wt{\mcl R}_{T,+}^\ep \geq N^\ep/4$ or $\# \wt{\mcl R}_{T,-}^\ep \geq N^\ep/4$. Assume that $\# \wt{\mcl R}_{T,+}^\ep \geq N^\ep/4$ (the other case is treated similarly, with the roles of $D_h$ and $\frk a_\ep^{-1} \wh D_h^\ep$ interchanged). We will show that $T$ is bounded above by a constant depending on $\zeta $ and the law of $D_h$. 

By~\eqref{eqn-compare-lower} of Lemma~\ref{lem-compare-subset} applied with $\wt{\mcl R}^\ep = \wt{\mcl R}^\ep_{T,+}$ and $U = B_2(0)$, say, there exists $\Csub  > 0$ such that with polynomially high probability as $\ep\rta 0$, 
\eqb  \label{eqn-use-compare-subset0}
D_h \left( B_{ \ep^{1-\zeta}}(z) , B_{ \ep^{1-\zeta}}(w) ; B_{2+ \ep^{1-\zeta}}(0) \right) 
\leq \Csub T^{-1}  \frk a_\ep^{-1} \wh D_h^\ep(z,w ; B_2(0) ) ,\quad\forall z,w\in B_2(0) ,
\eqe 
which implies that
\eqb \label{eqn-use-compare-subset}
D_h \left( \text{across $\BB A_{1+\ep^{1-\zeta} , 2 - \ep^{1-\zeta}}(0)$}  \right) 
\leq \Csub T^{-1}  \frk a_\ep^{-1} \wh D_h^\ep\left( \text{across $\BB A_{1,2}(0)$}   \right) . 
\eqe
By tightness across scales (Axiom~\ref{item-metric-coord}), there exists $S  > 0$, depending only on the law of $D_h$, such that whenever $\ep < 1/100$, say, we have 
\eqb \label{eqn-lqg-tight}
\BB P\left[ D_h \left( \text{across $\BB A_{1+\ep^{1-\zeta} , 2 - \ep^{1-\zeta}}(0)$}  \right)   \geq S^{-1} \right] \geq \frac34 .
\eqe 
By~\cite[Proposition 4.1]{dg-supercritical-lfpp} and Lemma~\ref{lem-localized-approx}, after possibly increasing $S$ we can arrange that also 
\eqb \label{eqn-lfpp-tight}
\BB P\left[  \frk a_\ep^{-1} \wh D_h^\ep\left( \text{across $\BB A_{1,2}(0)$}   \right)   \leq S  \right] \geq \frac34 .
\eqe 

By combining~\eqref{eqn-use-compare-subset}, \eqref{eqn-lqg-tight}, and~\eqref{eqn-lfpp-tight}, we obtain that with probability at least $1/4 - o_\ep(1)$ (with the rate of convergence of the $o_\ep(1)$ depending only on $\zeta $ and the law of $D_h$), we have 
\eqbn
S^{-1} \leq \Csub T^{-1} S ,
\eqen
i.e., $T \leq \Csub S^2$. Hence, if $\ep$ is small enough so that $1/4 - o_\ep(1) > 0$, then $T \leq \Csub S^2$. Therefore, the lemma statement holds with $\Cratio = \Csub S^2$. 
\end{proof}

As a consequence of Lemma~\ref{lem-good-ratios}, we obtain a version of Theorem~\ref{thm-lqg-lfpp} with $\wh D_h^\ep$ in place of $D_h^\ep$. 

\newcommand{\Ctrunc}{{C_0}}
 
\begin{prop} \label{prop-lqg-lfpp-trunc}
For each $\zeta \in (0,1) $, there exists and $\Ctrunc > 0$, depending only on $\zeta$ and the law of $D_h$, such that the following is true. 
Let $U\subset \BB C$ be a deterministic, connected, bounded open. With polynomially high probability as $\ep\rta 0$, 
\eqb \label{eqn-compare-upper-trunc}
\frk a_\ep^{-1} \wh D_h^\ep\left( B_{ \ep^{1-\zeta}}(z) , B_{ \ep^{1-\zeta}}(w) ; B_{ \ep^{1-\zeta}}(U) \right) 
\leq \Ctrunc D_h(z,w ; U)  ,\quad \forall z,w\in U
\eqe 
and
\eqb \label{eqn-compare-lower-trunc}
D_h \left( B_{ \ep^{1-\zeta}}(z) , B_{ \ep^{1-\zeta}}(w) ; B_{ \ep^{1-\zeta}}(U) \right) 
\leq \Ctrunc \frk a_\ep^{-1} \wh D_h^\ep(z,w ; U) ,\quad \forall z,w \in U . 
\eqe 
\end{prop}
\begin{proof} 
Let $\Cratio$ be as in Lemma~\ref{lem-good-ratios} and let $\wt{\mcl R}^\ep$ be the set of $r\in \mcl R^\ep$ for which~\eqref{eqn-good-ratios} holds. By Lemma~\ref{lem-good-ratios}, we have $\#\wt{\mcl R}^\ep \geq N^\ep/2$, so we can apply Lemma~\ref{lem-compare-subset} to get that with polynomially high probability as $\ep\rta 0$, the bounds~\eqref{eqn-compare-upper} and~\eqref{eqn-compare-lower} hold for our above choice of $\wt{\mcl R}^\ep$. We then use~\eqref{eqn-good-ratios} to bound the maximum and minimum appearing in~\eqref{eqn-compare-upper} and~\eqref{eqn-compare-lower} in terms of $\Cratio$. This gives the proposition statement with $\Ctrunc = \Csub \Cratio$. 
\end{proof}

\begin{proof}[Proof of Theorem~\ref{thm-lqg-lfpp}]
This is immediate from Lemma~\ref{lem-localized-approx} and Proposition~\ref{prop-lqg-lfpp-trunc}. 
\end{proof}

\subsection{Bounds for scaling constants and bi-Lipschitz equivalence}
\label{sec-scaling}

In this subsection we will prove Theorems~\ref{thm-c_r}, \ref{thm-bilip}, and~\ref{thm-a_eps}. 

Theorem~\ref{thm-lqg-lfpp} provides non-trivial bounds relating $\frk a_\ep^{-1} \wh D_h^\ep(z,w)$ and $D_h(z,w)$ whenever $|z-w|$ is of larger order than $ \ep^{1-\zeta}$.
From this and the scaling properties of LFPP, we get bounds for the ratios $\frac{r \frk a_{\ep/r}}{\frk c_r \frk a_\ep}$ whenever $r $ is much larger than $ \ep^{1-\zeta}$.  
These bounds will be the main input in the proofs of Theorems~\ref{thm-c_r} and~\ref{thm-a_eps}. 

\newcommand{\Cuniform}{{C_7}}

\begin{lem} \label{lem-ratio-uniform}
There is a constant $\Cuniform  > 1$, depending only on $\zeta$ and the law of $D_h$, such that the following is true. For each $R \geq 1$, there exists $\ep_* = \ep_*(R,\zeta) > 0$ such that for each $\ep \in (0,\ep_*]$ and each $r\in [100\ep^{1-\zeta} , R]$,  
\eqb \label{eqn-ratio-uniform}
\Cuniform^{-1} \leq \frac{r \frk a_{\ep/r}}{\frk c_r \frk a_\ep} \leq \Cuniform .
\eqe 
\end{lem}

We emphasize the distinction between Lemmas~\ref{lem-good-ratios} and~\ref{lem-ratio-uniform}: the former gives bounds for the ratios $\frac{r \frk a_{\ep/r}}{\frk c_r \frk a_\ep}$ which hold for \emph{most} $r\in \mcl R^\ep \subset [\ep,\ep^{1-\zeta}]$ whereas the latter gives bounds for \emph{all} $r \in [100\ep^{1-\zeta}, 1]$.

\begin{proof}[Proof of Lemma~\ref{lem-ratio-uniform}]
We will find $\Cuniform$ and $\ep_*$ such that the upper bound in~\eqref{eqn-ratio-uniform} holds. The lower bound is obtained via a similar argument, with the roles of $\frk a_\ep^{-1} \wh D_h^\ep$ and $D_h$ interchanged.

Fix $R \geq  1$.  
By Theorem~\ref{thm-lqg-lfpp} applied with $U = B_{3R}(0)$, there exists $\Cmain > 0$, depending only on $\zeta$ and the law of $D_h$, such that with polynomially high probability as $\ep\rta 0$ (the rate of convergence depends on $R,\zeta$, and the law of $D_h$), we have
\eqbn
\frk a_\ep^{-1} \wh D_h^\ep\left( B_{ \ep^{1-\zeta}}(z) , B_{ \ep^{1-\zeta}}(w) ; B_{3R +  \ep^{1-\zeta}}(0) \right) 
\leq \Cmain  D_h(z,w ; B_{3R}(0) )  ,\quad \forall z,w\in B_{3R}(0) .
\eqen 
By applying this last inequality to points on the inner and outer boundaries of $\BB A_{r-\ep^{1-\zeta}, 2r+\ep^{1-\zeta}}(0)$, we get that with polynomially high probability as $\ep\rta 0$ that for each $r \in [100 \ep^{1-\zeta} , R]$, 
\eqb  \label{eqn-use-compare-upper'}
\frk a_\ep^{-1} \wh D_h^\ep\left(  \text{across $\BB A_{r   , 2r }(0)$} \right) 
\leq \Cmain  D_h \left(  \text{across $\BB A_{r -  \ep^{1-\zeta}   , 2r +  \ep^{1-\zeta}  }(0)$} \right)  . 
\eqe 

Using tightness across scales (Axiom~\ref{item-metric-coord}) and the tightness of LFPP crossing distances (\cite[Proposition 4.1]{dg-supercritical-lfpp}), as in the proof of Lemma~\ref{lem-compare-event-prob}, we find that there exists $S   > 0$, depending only on the law of $D_h$, such that for each $r\in [100\ep^{1-\zeta},R]$, 
\allb \label{eqn-ratio-uniform-compare}
\BB P\left[ \frk a_\ep^{-1} \wh D_h^\ep\left(  \text{across $\BB A_{r   , 2r }(z)$} \right) \geq S^{-1} \frac{r \frk a_{\ep/r}}{\frk a_\ep} e^{\xi h_r(0)} \right] &\geq \frac34 \quad \text{and} \notag\\
\BB P\left[ D_h \left(  \text{across $\BB A_{r -  \ep^{1-\zeta}   , 2r +  \ep^{1-\zeta}  }(z)$} \right)   \leq S \frk c_r e^{\xi h_r(0)} \right] &\geq \frac34  .
\alle

By combining~\eqref{eqn-use-compare-upper'} and~\eqref{eqn-ratio-uniform-compare}, we get that for each $r\in [100\ep^{1-\zeta},R]$, it holds with probability at least $1/4 - o_\ep(1)$ (with the rate of convergence depending only on $R, \zeta $ and the law of $D_h$) that 
\eqbn
S^{-1} \frac{r \frk a_{\ep/r}}{\frk a_\ep} \leq \Cmain S \frk c_r .
\eqen
Hence, if $\ep$ is small enough so that $1/4 - o_\ep(1) > 0$, then
\eqbn
\frac{r \frk a_{\ep/r}}{\frk c_r \frk a_\ep} \leq \Cmain S^2 .
\eqen
This gives the upper bound in~\eqref{eqn-ratio-uniform} with $\Cuniform = \Cmain S^2$. As noted above, the lower bound is proven similarly. 
\end{proof}

We will deduce our bounds for $\frk c_r$ and $\frk a_\ep$ (Theorems~\ref{thm-c_r} and~\ref{thm-a_eps}) from Lemma~\ref{lem-ratio-uniform} together with elementary deterministic arguments. 
For the proof of Theorem~\ref{thm-c_r}, we need the following classical lemma, which tells us that if a sequence $\{x_n\}_{n\in\BB N}$ is both subadditive and superadditive, up to a constant additive error, then $x_n/n$ converges to a limit and one can bound the rate of convergence. 

\begin{lem}[Subadditive rate lemma] \label{lem-subadd}
Let $\{x_n\}_{n\in\BB N}$ be a sequence of real numbers and assume that there is a $c > 0$ such that
\eqb  \label{eqn-subadd}
x_n + x_m - c \leq x_{n+m} \leq x_n + x_m + c ,\quad\forall n,m \in \BB N .
\eqe 
Then there is an $\alpha > 0$ such that 
\eqb  \label{eqn-subadd-rate}
|x_n / n - \alpha | \leq c / n ,\quad \forall n\in\BB N .
\eqe
\end{lem}

Lemma~\ref{lem-subadd} follows from the proof of~\cite[Lemma 1.9.1]{steele} (applied with $x_n/c$ in place of $x_n$). The statement of~\cite[Lemma 1.9.1]{steele} gives $|x_n/n-\alpha| \leq c$ instead of~\eqref{eqn-subadd-rate}, but the proof shows that in fact~\eqref{eqn-subadd-rate} holds. 

We will also need a basic a priori estimate comparing the scaling constants $\frk c_r$ for different values of $r$.

\begin{lem} \label{lem-constant-a-priori}
For each $K > 1$, there exists $C > 1$, depending on $K$ and the law of $D_h$, such that $C^{-1} \frk c_r \leq \frk c_{r'} \leq C \frk c_r$ whenever $r>0$ and $r'\in [K^{-1} r , K r]$. 
\end{lem}
\begin{proof}
Fix a Euclidean annulus $A\subset\BB C$. We can find finitely many Euclidean annuli $A_1,\dots,A_k$ such that for each $s \in [K^{-1} , K]$, there exists $j\in [1,k]_{\BB Z}$ such that $s A$ is contained in $A_j$ and disconnects the inner and outer boundaries of $A_j$ (we make the aspect ratios of the $A_j$'s larger than the aspect ratio of $A$). Similarly, we can find finitely many Euclidean annuli $A_1',\dots,A_k'$ (whose aspect ratios will be smaller than the aspect ratio of $A$) such that for each $s \in [K^{-1} , K]$, there exists $j'\in [1,k]_{\BB Z}$ such that $A_{j'}'$ is contained in $sA$ and disconnects the inner and outer boundaries of $s A$. We have
\alb
&D_h(\text{around $A_j$}) \leq D_h(\text{around $s A$}) \leq D_h(\text{around $A_{j'}'$} ) 
\quad \text{and} \notag\\
&\qquad D_h(\text{across $A_{j'}'$}) \leq D_h(\text{across $s A$}) \leq D_h(\text{across $A_j$} ) .
\ale
From this and Axiom~\ref{item-metric-coord}, applied to each of the annuli $A_1,\dots,A_k$ and $A_1',\dots,A_k'$, we get that the random variables
\eqb \label{eqn-sup-tight}
\frk c_r^{-1} e^{-\xi h_r(0)} \sup_{r' \in [K^{-1} r , K r]} D_h(\text{around $r' A$}) 
\eqe 
are tight; and the same holds if we replace the sup by an inf and take the reciporicals of the random variables; and/or we replace ``across" by ``around". 
Furthermore, since $t\mapsto h_{e^{-t}}(0)$ is a standard Brownian motion (see the calculations in~\cite[Section 3.1]{shef-kpz}), we get that the random variables
\eqb \label{eqn-circle-avg-tight}
\sup_{r' \in [K^{-1} r , K r]} \exp\left( \xi |h_{r'}(0) - h_r(0)| \right) 
\eqe 
are tight. Combining~\eqref{eqn-sup-tight} and~\eqref{eqn-circle-avg-tight} shows that the random variables
\eqbn
\frk c_r^{-1}\sup_{r' \in [K^{-1} r , K r]}  e^{-\xi h_{r'}(0)}  D_h(\text{around $r' A$}) 
\eqen
are tight; and the same holds if we replace the sup by an inf and take the reciporicals of the random variables; and/or we replace ``across" by ``around". 
Consequently, Axiom~\ref{item-metric-coord} holds with $\frk c_r$ replaced by the scaling factors $\wt{\frk c}_r$ which equals $\frk c_{K^n}$ whenever $r \in [K^n, K^{n+1}]$. 
By Remark~\ref{remark-constant-multiple}, we get that there is a constant $C>1$ such that $C^{-1} \frk c_{K^n} \leq \frk c_r \leq C \frk c_{K^n}$ whenever $r \in [K^n ,K^{n+1}]$. This implies the lemma statement with $C^2$ in place of $C$. 
\end{proof}

\begin{proof}[Proof of Theorem~\ref{thm-c_r}]
Throughout the proof, we assume that all implicit constants in $\asymp$ depend only on the law of $D_h$.  
Let $r ,s > 0$. By three applications of Lemma~\ref{lem-ratio-uniform} applied with $\zeta = 1/2$, say, if $\ep $ is sufficiently small (depending on $r$ and $s$) then
\eqb  \label{eqn-submult}
\frk c_r \frk c_s \asymp \frac{r \frk a_{\ep/r}}{\frk a_\ep} \times \frac{s \frk a_{\ep / (s r)}}{\frk a_{\ep/r}} = \frac{s r \frk a_{\ep / (s r)}}{\frk a_\ep} \asymp \frk c_{s r} .
\eqe 

For $n\in\BB N$, write $x_n = \log \frk c_{2^{-n}}$. 
By taking $s$ and $r$ to be powers of 2 and taking the log of both sides of~\eqref{eqn-submult}, we obtain that $\{x_n\}_{n\in\BB N}$  satisfies~\eqref{eqn-subadd}. 
Therefore, Lemma~\ref{lem-subadd} implies that there exists $\alpha > 0$ such that $|x_n / n - \alpha | \leq c/n$ for all $n$, equivalently 
\eqbn
\frk c_{2^{-n}} \asymp 2^{-\alpha n} .
\eqen
 By~\cite[Proposition 4.2]{dg-supercritical-lfpp}, for $0 < \ep < r$, we have 
\eqbn
\frac{r \frk a_{\ep/r}}{  \frk a_\ep} = r^{\xi Q +o_r(1)} ,
\eqen
with the rate of the convergence of the $o_r(1)$ depending only on the law of $D_h$. By combining this with Lemma~\ref{lem-ratio-uniform}, we infer that $\alpha = \xi Q$. This gives~\eqref{eqn-c_r} when $r$ is a negative power of 2. 

The case of a general choice of $r\in (0,1]$ follows from the case when $r =2^{-n}$ together with Lemma~\ref{lem-constant-a-priori}. To treat the case when $r  > 1$, we apply~\eqref{eqn-submult} with $s = 1/r < 1$ to get
\eqbn
\frk c_r \asymp \frk c_1 / \frk c_{1/r} \asymp r^{\xi Q} .
\eqen  
\end{proof}

\begin{proof}[Proof of Theorem~\ref{thm-a_eps}]
Let $\ep_* > 0$ be as in Lemma~\ref{lem-ratio-uniform} with $\zeta = 1/4$ and $R=1$. By possibly shrinking $\ep_*$, we can arrange that also $100 \ep^{3/4} \leq \ep^{1/2}$ for each $\ep \in (0,\ep_*]$. Lemma~\ref{lem-ratio-uniform} (applied with $\zeta = 1/4$) combined with Theorem~\ref{thm-c_r} implies that there is a constant $A = A(\xi) > 1$ (in particular, $A = \Clqg \Cuniform$) such that if $0  <\ep \leq \ep_*$ and $r\in [\ep^{1/2} ,   1 ]$, then 
\eqb \label{eqn-use-ratio-uniform}
A^{-1} r^{\xi Q - 1} \leq \frac{ \frk a_{\ep/r}}{ \frk a_\ep} \leq A r^{\xi Q - 1} .
\eqe 
After possibly increasing $A$, we can remove the constraint that $\ep \leq \ep_*$.
  
For $k\in\BB N_0$, we apply~\eqref{eqn-use-ratio-uniform} with $\ep = 2^{-2^{ k}}$ and $r = 2^{-2^k}/\delta$ to find that
\eqb \label{eqn-a_eps-gen}
A^{-1} \delta^{1-\xi Q} 2^{(1-\xi Q) 2^k} \leq \frac{ \frk a_{\delta }}{ \frk a_{2^{-2^k}}} \leq A \delta^{1-\xi Q} 2^{(1-\xi Q) 2^k} ,\quad \forall \delta \in [2^{-2^k} , 2^{-2^{k-1}}] .
\eqe

In particular, taking $\delta = 2^{-2^{k-1}}$ gives
\eqb  \label{eqn-a_eps-k}
A^{-1}  2^{(1-\xi Q) 2^{k-1} } \leq \frac{ \frk a_{2^{-2^{k-1}}}}{ \frk a_{2^{-2^k}}} \leq A 2^{(1-\xi Q) 2^{k-1}}  .
\eqe
We apply this inequality with $j$ instead of $k$, then multiply over all $j = 1,\dots,k$ to get
\eqb  \label{eqn-a_eps-mult}
A^{-k}  2^{(1-\xi Q) (2^k-1) } \leq \frac{ \frk a_{1/2}}{ \frk a_{2^{-2^k}}} \leq  A^k 2^{(1-\xi Q) (2^k-1)}   .
\eqe
Re-arranging shows that there is a constant $C = C(\xi) > 0$ such that
\eqb \label{eqn-a_eps-bound-k}
C^{-1} A^{-k} 2^{-(1-\xi Q) 2^k} \leq  \frk a_{2^{-2^k}} \leq C A^k 2^{-(1-\xi Q) 2^k} ,\quad\forall k\in\BB N 
\eqe
where here we absorbed $\frk a_{1/2}$ into $C$. 

For a given $\delta\in (0,1/2]$, choose $k\in\BB N$ such that $\delta \in  [2^{-2^k} , 2^{-2^{k-1}}]$. Note that
\eqb \label{eqn-a_eps-log}
k \in [\log_2\log_2 \delta^{-1} , \log_2 \log_2 \delta^{-1} +1] .
\eqe 
By~\eqref{eqn-a_eps-gen} and~\eqref{eqn-a_eps-bound-k}, 
\eqb \label{eqn-a_eps-bound}
C^{-1} A^{-k-1} \delta^{1-\xi Q} \leq  \frk a_\delta \leq C A^{k+1} \delta^{1-\xi Q}  .
\eqe
By~\eqref{eqn-a_eps-log}, $A^{k+1}$ is bounded above by a $\xi$-dependent constant times $(\log\delta^{-1})^b$ for some $b = b(\xi) > 0$. 
Thus~\eqref{eqn-a_eps-bound} implies~\eqref{eqn-a_eps}.
\end{proof}

\begin{remark} \label{remark-polylog}
Our proof does not yield better than polylogarithmic upper and lower bounds for $\frk a_\ep / \ep^{1-\xi Q}$. 
Indeed, the estimate~\eqref{eqn-use-ratio-uniform} for $r\in [\ep^{1/2},1]$ is still satisfied, e.g., if $\frk a_\ep = (\log \ep^{-1})^b \ep^{1-\xi Q}$ for some $b\in\BB R$ (with the constant $A$ depending on $b$). 
In order to get better than polylogarithmic bounds, we would need to improve on~\eqref{eqn-use-ratio-uniform} so that either it holds for all $r \in [\phi(\ep) \ep , 1]$, where $\lim_{\ep\rta 0} \log \phi(\ep) / \log\ep = 0$; or so that it holds with $A$ replaced by something of order $1+o_\ep(1)$. 
Either of these improvements would require non-trivial new ideas.
\end{remark}

\begin{proof}[Proof of Theorem~\ref{thm-bilip}]
This can be deduced from Theorem~\ref{thm-c_r} and a generalization of the bi-Lipschitz equivalence criterion from~\cite[Theorem 1.6]{local-metrics}. However, we will instead give a more self-contained proof.

Let $U\subset\BB C$ be a deterministic, connected, bounded open set. 
We apply Theorem~\ref{thm-lqg-lfpp} (with $\zeta = 1/2$) to compare each of $D_h$ and $\wt D_h$ to $\frk a_\ep^{-1} \wh D_h^\ep$. We obtain that there is a deterministic constant $\Cbilip >0$, depending only on the laws of $D_h$ and $\wt D_h$, such that with probability tending to 1 as $\ep\rta 0$,  
\eqb \label{eqn-bilip-upper}
\wt D_h\left( B_{ \ep^{1/2}}(z) , B_{ \ep^{1/2}}(w) ; B_{ \ep^{1/2}}(U) \right) 
\leq \Cbilip D_h(z,w ; U)  ,\quad \forall z,w\in U 
\eqe 
and
\eqb \label{eqn-bilip-lower}
D_h \left( B_{ \ep^{1/2}}(z) , B_{ \ep^{1/2}}(w) ; B_{ \ep^{1/2}}(U) \right) 
\leq \Cbilip  \wt D_h(z,w ; U) ,\quad \forall z,w\in U .
\eqe 
In particular, $\Cbilip$ is the product of the constants appearing in Theorem~\ref{thm-lqg-lfpp} for $D_h$ and $\wt D_h$, respectively. 
Shrinking $\ep$ makes the conditions~\eqref{eqn-bilip-upper} and~\eqref{eqn-bilip-lower} stronger. Since these conditions hold with probability tending to 1 as $\ep\rta 0$, we infer that a.s.\ there is a random $\ep_* = \ep_*(U) > 0$ such that \eqref{eqn-bilip-upper} and~\eqref{eqn-bilip-lower} hold for each $\ep\leq \ep_*$. 

Now let $\{U_n\}_{n\in\BB N}$ be an increasing family of bounded open sets whose union is all of $\BB C$. From the preceding paragraph, we infer that a.s.\ there is a random sequence of positive numbers $\{\ep_n\}_{n\in\BB N}$ such that for each $n\in\BB N$, the conditions~\eqref{eqn-bilip-upper} and~\eqref{eqn-bilip-lower} hold with $U=U_n$ for each $\ep \leq \ep_n$. 

Almost surely, every $D_h$-bounded set is Euclidean-bounded~\cite[Lemma 3.12]{pfeffer-supercritical-lqg}. Consequently, it is a.s.\ the case that for any two distinct points $z,w\in \BB C$ which are non-singular for $D_h$, there exists $n\in\BB N$ such that every path from $z$ to $w$ whose $D_h$-length is at most $2D_h(z,w)$ is contained in $ U_n $. This implies that $D_h(z,w;U_n) = D_h(z,w)$. By combining this with the preceding paragraph, we find that for each $\ep \in (0,\ep_n]$, 
\eqbn
\wt D_h\left( B_{ \ep^{1/2}}(z) , B_{  \ep^{1/2}}(w)  \right) 
\leq \wt D_h\left( B_{ \ep^{1/2}}(z) , B_{  \ep^{1/2}}(w) ; B_{ \ep^{1/2}}(U_n) \right) 
\leq \Cbilip D_h(z,w ). 
\eqen
Since $\wt D_h$ is lower semicontinuous, if we take the liminf of the left side of this inequality as $\ep\rta 0$, we obtain $\wt D_h(z,w) \leq \Cbilip D_h(z,w)$. This holds for any two points $z,w\in \BB C$ which are non-singular for $D_h$. If either $z$ or $w$ is a singular point for $D_h$, then $D_h(z,w) = \infty$ so $\wt D_h(z,w) \leq \Cbilip D_h(z,w)$ vacuously. We thus obtain the upper bound in~\eqref{eqn-bilip}. The lower bound is obtained similarly, with the roles of $D_h$ and $\wt D_h$ interchanged.
\end{proof}

\bibliography{cibib}
\bibliographystyle{hmralphaabbrv}

\end{document}